\documentclass[12pt]{article}

\usepackage{subfigure}

\usepackage{latexsym}
\usepackage{graphics}
\usepackage{amsmath}
\usepackage{xspace}
\usepackage{amssymb}
\usepackage{psfrag}
\usepackage{epsfig}
\usepackage{amsthm}
\usepackage{pst-all}

\usepackage{color}

\definecolor{Red}{rgb}{1,0,0}
\definecolor{Blue}{rgb}{0,0,1}
\definecolor{Olive}{rgb}{0.41,0.55,0.13}
\definecolor{Yarok}{rgb}{0,0.5,0}
\definecolor{Green}{rgb}{0,1,0}
\definecolor{MGreen}{rgb}{0,0.8,0}
\definecolor{DGreen}{rgb}{0,0.55,0}
\definecolor{Yellow}{rgb}{1,1,0}
\definecolor{Cyan}{rgb}{0,1,1}
\definecolor{Magenta}{rgb}{1,0,1}
\definecolor{Orange}{rgb}{1,.5,0}
\definecolor{Violet}{rgb}{.5,0,.5}
\definecolor{Purple}{rgb}{.75,0,.25}
\definecolor{Brown}{rgb}{.75,.5,.25}
\definecolor{Grey}{rgb}{.5,.5,.5}

\newcommand{\E}[1]{\mathbb{E}\!\left[#1\right]}

\newcommand{\ignore}[1]{\relax}

\newtheorem{theorem}{Theorem}[section]
\newtheorem{remark}[theorem]{Remark}
\newtheorem{lemma}[theorem]{Lemma}

\newtheorem{proposition}[theorem]{Proposition}

\newtheorem{claim}[theorem]{Claim}

\newtheorem{definition}[theorem]{Definition}

\definecolor{Red}{rgb}{1,0,0}
\definecolor{Blue}{rgb}{0,0,1}
\definecolor{Olive}{rgb}{0.41,0.55,0.13}
\definecolor{Green}{rgb}{0,1,0}
\definecolor{MGreen}{rgb}{0,0.8,0}
\definecolor{DGreen}{rgb}{0,0.55,0}
\definecolor{Yellow}{rgb}{1,1,0}
\definecolor{Cyan}{rgb}{0,1,1}
\definecolor{Magenta}{rgb}{1,0,1}
\definecolor{Orange}{rgb}{1,.5,0}
\definecolor{Violet}{rgb}{.5,0,.5}
\definecolor{Purple}{rgb}{.75,0,.25}
\definecolor{Brown}{rgb}{.75,.5,.25}
\definecolor{Grey}{rgb}{.5,.5,.5}
\definecolor{Pink}{rgb}{1,0,1}
\definecolor{DBrown}{rgb}{.5,.34,.16}
\definecolor{Black}{rgb}{0,0,0}

\usepackage{float}

\usepackage{float}
\author{
{\sf David Gamarnik}\thanks{MIT; e-mail: {\tt gamarnik@mit.edu}. Research supported  by the NSF grants CMMI-1335155.}
\and
{\sf Ilias Zadik}\thanks{MIT; e-mail: {\tt izadik@mit.edu}}
}

\begin{document}

\title{Sparse High-Dimensional Linear Regression. \\ Algorithmic Barriers and a Local Search Algorithm.}
\date{}

\maketitle

\begin{abstract}
We consider a sparse high dimensional regression model where the goal is to recover a $k$-sparse unknown vector $\beta^*$ from $n$ noisy linear observations of the form $Y=X\beta^*+W \in \mathbb{R}^n$ where $X \in \mathbb{R}^{n \times p}$ has iid $N(0,1)$ entries and $W \in \mathbb{R}^n$ has iid $N(0,\sigma^2)$ entries. Under certain assumptions on the parameters, an intriguing assymptotic gap appears between the minimum value of $n$, call it $n^*$, for which the recovery is information theoretically possible, and the minimum value of $n$, call it $n_{\mathrm{alg}}$, for which an efficient algorithm is known to provably recover $\beta^*$. In \cite{gamarnikzadik} it was conjectured that the gap is not artificial, in the sense that for sample sizes $n \in [n^*,n_{\mathrm{alg}}]$ the problem is algorithmically hard. 
 
 We support this conjecture in two ways. Firstly, we show that the optimal solution of the LASSO provably fails to $\ell_2$-stably recover the unknown vector $\beta^*$ when $n \in [n^*,c n_{\mathrm{alg}}]$, for some sufficiently small constant $c>0$. Secondly, we establish that $n_{\mathrm{alg}}$, up to a multiplicative constant factor, is a phase transition point for the appearance of a certain Overlap Gap Property (OGP) over the space of $k$-sparse vectors. The presence of such an Overlap Gap Property phase transition, which originates in statistical physics, is known to provide evidence of an algorithmic hardness. Finally we show that if $n>C n_{\mathrm{alg}}$ for some large enough constant $C>0$, a very simple algorithm based on a local search improvement rule is able both to $\ell_2$-stably recover the unknown vector $\beta^*$ and to infer correctly its support, adding it to the list of provably successful algorithms for the high dimensional linear regression problem.
\end{abstract}

\section{Introduction} 
We consider the following high-dimensional regression model. $n$ noisy linear observations of a vector  $\beta^* \in \mathbb{R}^p$ of the form $Y=X\beta^*+W$ are observed, for some $X \in \mathbb{R}^{ n \times p}$ and $W \in \mathbb{R}^n$. Given these observations, and the knowledge of $X$, but not of $W$, the vector $\beta^*$ needs to be inferred. The goal is to infer $\beta^*$ with the minimum number of observations $n$. Throughout the paper we call $X$ the measurement matrix and $W$ the noise vector. 

We are interested in the high dimensional setting where $n$ is order of magnitude less than $p$, and they both diverge to infinity. High-dimensionality is motivated by various statistical applications over the last decade for example in the field of radiology and biomedical imaging (see e.g. \cite{DonohoMRI} and references therein) and in the field of genomics  \cite{BickelGenome}, \cite{JASAgenomics} and it has been very common in the literature during the past decade \cite{wainwright2009information},\cite{BertsimasRegression},\cite{johnstone2009consistency}. This, in principle makes the recovery problem impossible even if $W=0$, as in this case the underlying linear system is underdetermined. This difficulty is commonly adressed by imposing a sparsity assumption on the vector $\beta^*$. More specifically, we say that the unknown vector $\beta^*$ is $k$-sparse (exactly $k$-sparse) if it has at most $k$ non-zero coordinates (exactly $k$-nonzero coordinates). The sparsity is a very useful assumption in applications,  for example in compressed sensing~\cite{candes2005decoding}, \cite{donoho2006compressed} , biomedical imaging \cite{MRImed}, \cite{DonohoMRI}  and sensor networks \cite{CSwireless}, \cite{CSsensor}, but also in theory \cite{donoho2006compressed}. For our purposes we assume that the value of $k$ is known for all the results. Furthermore, we are interested in both the case that $\beta^*$ is generally $k$-sparse and also $\beta^*$ is  exactly $k$-sparse, and we make clear on the statement of each result which assumption we are making on $\beta^*$.

We also make probabilistic assumptions on $X$ and $W$. we assume that each row of $X$ is generated as an iid sample from an isotropic $\mathcal{N}\left(0,\Sigma\right),$ where we take $\Sigma=I_{p}$. Note that the Gaussianity of the data rows is, in a standard way, justified from the Central Limit theorem and is very common in the literature \cite{Casto11}, \cite{Lucas17}, \cite{Van13}, \cite{Tony18}, \cite{wainwright2009sharp}, \cite{wainwright2009information},\cite{wang2010information}. Furthermore, the case $\Sigma=I_p$, which can be considered unrealistic from an applied point of view, has been considered broadly in the literature as an idealized assumption which allows broader technical development which can usually be generalized  \cite{Casto11}, \cite{Lucas17}, \cite{wainwright2009sharp}, \cite{wainwright2009information}, \cite{wang2010information}. We assume also that $W$ consists of iid $N(0,\sigma^2)$ entries for some $\sigma^2>0$, which is a standard assumption in the statistics literature \cite{wainwright2009information},  \cite{wang2010information} and \cite{donoho2006counting}.

In this paper, we focus on two notions of recovery for the unknown vector $\beta^*$. Firstly, we consider the notion of \textit{support recovery} \cite{SignDen},\cite{wainwright2009information},\cite{chai} the task of finding an estimator vector $\hat{\beta}$ with support approximately equal to the support of $\beta^*$, where the Hamming distance is the underlying metric. support recovery is also known in the literature as sparsity pattern recovery task \cite{Galen13}, variable selection (see \cite{Ed2012} and references therein) or model selection \cite{Model93}, \cite{Aos06}. Secondly, we consider the notion of \textit{$\ell_2$ stable recovery} \cite{candes},\cite{davies} the task of finding an estimator vector $\hat{\beta}$ such that $\| \hat{\beta}-\beta^*\|_2 \leq C \sigma$, for some $C>0$. In words, the estimator vector is close to the unknown vector in the $\ell_2$ distance up to the level of noise.  Because of our probabilistic assumptions on $X,W$ both recoveries are desired to occur with high probability \textit{(w.h.p.)} with respect to the randomness of $X,W$, that is with probability tending to one, as $n,p,k \rightarrow +\infty$. Here the limit is taken under certain assumptions on the relation between the parameters $n,p,k,\sigma^2$ that will be stated explicitly in the next sections. Finally, it is important to point out that, similarly with \cite{gamarnikzadik}, we generally think of the case of small enough sparsity so that the logarithm of $k$ is much smaller than $\log p$, and hence \textit{the sparsity level is sublinear in the feature size $p$}. On the other hand, some of our results,  such as the ones described in subsection 2.2. below, apply under the more general condition $k \leq p/3$.

Various efficient algorithms have been proven to recover w.h.p. the vector $\beta^*$ in the two notions of recovery we mention above, but always under the assumption that $n \geq C k \log p$ for some universal constant $C>0$. For this reason we define $n_{\mathrm{alg}}:=k \log p$. Specifically, with respect to support recovery, if $n \geq (1+\epsilon) 2 n_{\mathrm{alg}}= (1+\epsilon) 2k \log p$, for some $\epsilon>0$ it is proven by Wainwright and Cai et al in \cite{wainwright2009sharp} and \cite{chai} respectively that the optimal solution of an associated $\ell_1$-constrained quadratic optimization formulation called LASSO \begin{align}\label{LassoSt} \text{LASSO}_\lambda: \min_{\beta \in \mathbb{R}^p} n^{-1}\|Y-X\beta\|^2_2+\lambda \|\beta\|_1 
\end{align}for appropriately chosen tuning parameter $\lambda>0$, and that the output of a simple greedy algorithm called Orthogonal Matching Pursuit, both recover exactly the support of $\beta^*$ w.h.p.
With respect to $\ell_2$ stable recovery, the tighter results known are for the performance of LASSO and of a linear program called the Dantzig selector \cite{candes2007,bickel2009}, both of which requires $n \geq Cn_{\mathrm{alg}}$. More specifically, an easy corollary of the seminar work by Bickel, Ritov and Tsybakov \cite{bickel2009} applied to $X$ with Gaussian iid entries implies that as long as  $n \geq C n_{\mathrm{alg}}=Ck\log p$ for some sufficiently large constant $C>0$ if $\lambda=A\sigma \sqrt{\log p/n}$ the optimal solution $\hat{\beta}_{\mathrm{LASSO},\lambda}$ of  $\mathrm{LASSO}_{\lambda}$ satisfies for some constant $c>0$, $\|\hat{\beta}-\beta^*\|_2 \leq c\sigma$ w.h.p., i.e. it $\ell_2$-stably recovers the vector $\beta^*$, w.h.p. Tighter results for the performance on LASSO and the constants $c,C$ are established in the literature (see \cite{Chris13} and references therein), yet they do not apply in the regime where the sparsity is sublinear to the feature size $p$, which as it is explained above, is the main focus of this work.


However, in the case $n \leq c n_{\mathrm{alg}}=ck\log p$ where $c>0$ is a small constant, fewer results are known. In the context of support recovery when $n \leq c n_{\mathrm{alg}}=ck\log p$ where $c>0$ is a small constant, it is established in \cite{wang2010information} that if $n \leq cn_{\mathrm{alg}}$ for some constant $c>0$, it becomes information theoretic impossible to recover the support of any $k$-sparse vector $\beta^*$. In their setting, however, the entries of $\beta^*$ are allowed to take arbitrary small non-zero values. Specifically the absolute values of the non-zero entries are allowed to be of the same order as $\frac{1}{\sqrt{k}}$. This small magnitude of the non-zero entries naturally leads to larger sample complexity. The situation changes though if the non-zero entries of $\beta^*$ are, in absolute values, bounded away from zero by a constant. For example assuming $\beta^*$ is binary, that is $\beta^* \in \{0,1\}^p$, if we have also $k \leq \min \{1,\sigma^2\}\exp \left(C\sqrt{\log p} \right)$ for some $C>0$ and $\sigma^2 $ is much smaller than $k$, the tight information theoretic limit for recovering all but a negligible fraction of the support of $\beta^*$ is known to be equal to $n^*:=2k \log p/\log \left(\frac{2k}{\sigma^2}+1\right)$ w.h.p. which is asymptotically less than $ k \log p$, as established by Gamarnik and Zadik in the conference paper \cite{gamarnikzadik}.  The techniques of this paper are expected to generalize from the binary case to the case where $\beta^*$ is arbitrary with $|\beta^*|_{\mathrm{min}} \triangleq \min\{ |\beta^*_{i}| \big{|} \beta^*_i \not = 0\} \geq 1$.  Here 1 can be replaced with an arbitrary constant that does depends on $n,p,k,\sigma^2$. Rad in \cite{Rad2011} has independently partially proven a similar positive part of this result; he established that for some large enough constant $C>0$, if $n>Cn^*$ then one can recover exactly the support of $\beta^*$ and under the general condition $|\beta^*|_{\mathrm{min}} \geq 1$. To the best of our knowledge, no computationally efficient estimator is known to accurately recovering the support of $\beta^*$ for this number of samples. The main technical reason is that most of the results in the literature usually require a structural property to hold for $X$, such as the Restricted Isometry Property (RIP), Restricted Eigenvalue Property (RE) or Uniform Uncertainty Property  (UUP)  (see e.g. \cite{vandegeer2009}, \cite[Chapter 11]{Book15} and references therein), which is not known to hold for a matrix $X$ with iid standard Gaussian entries with less than $k \log p$ rows.  This abscence of computationally efficient results for support recovery  naturally brings the question of whether, under the assumption $|\beta^*|_{\min} \geq 1$,  efficient algorithms can be proven to recover the support of $\beta^*$ when $n^* \leq n \leq cn_{\mathrm{alg}}$ for some small constant $c>0$. This question is the main focus of this paper.

In the case  $n \leq c n_{\mathrm{alg}}=ck\log p$ where $c>0$ is a small constant, even fewer results are known for $\ell_2$-stably recovery.  In the case $\beta^*$ is binary, the result of Rad \cite{Rad2011} implies that exact recovery of $\beta^*$ is possible with order $n^*$ samples.  Hence, as the vector can be recovered exactly, it  can be also trivially $\ell_2$-stably recovered, granting order $n^*$ samples sufficient for recovery. To the best of our knowledge, no general information-theoretic result is known in the case $|\beta^*|_{\min} \geq 1$. For computationally efficient recovery, the most relevant result for our setting and $\ell_2$-stable recovery when  $n \leq c n_{\mathrm{alg}}=ck\log p$ , appears in \cite{Chris13} and establishes that LASSO fails to $\ell_2$-stably recover $\beta^*$ in this regime. Yet, the analysis in \cite{Chris13} trivializes the moment we assume that the sparsity level is sublinear to the feature size, i.e. $k/p \rightarrow 0$. This makes the study of $\ell_2$ stably recovery when $k/p \rightarrow 0$a wide open research direction. In particular, it  leaves open the question of whether LASSO or any other efficient estimator work well in this regime. In this paper, we present the first, to the best of our knowledge, (negative) result on LASSO in the regime $n^* \leq n \leq c n_{\mathrm{alg}}$ and when the sparsity level $k$ is sublinear in $p$.

It should be noted that in the more restrictive case that either the $k$-sparse vector $\beta^*$ is known to satisfy a structural ordering property called power allocation, or the matrix $X$ is assumed to be spatially coupled - a statistical physics notion-, variants of a computationally efficient scheme called Approximate Message Passing have been proven to succesfully work in the regime  $n^* \leq n \leq cn_{\mathrm{alg}}$. \cite{Barron12, Barron14, Rush17, Barbier14, DonohoAMP}. Nevertheless, we are interested here in the general case for $k$-sparse $\beta^*$ where either the $\beta^*$ is binary or it satisfies $|\beta^*|_{\mathrm{min}} \geq 1$, where no ordering is assumed to be known a priori to the statistician, and the case where $X$ has i.i.d. Gaussian entries where spatial coupling does not hold. In this general case, to the best of our knowledge, no results establishes that Approximate Message Passing works when $n^* \leq n \leq cn_{\mathrm{alg}}$ and there is a strong belief that any computationally efficient scheme fails \cite{gamarnikzadik}, as we describe in the following paragraph.

In \cite{gamarnikzadik} the authors conjecture that in the regime $n^* \leq n \leq cn_{\mathrm{alg}}$ the support recovery problem with a general $k$-sparse $\beta^*$ with $|\beta^*|_{\mathrm{min}} \geq 1$ is algorithmically hard, in the sense that there is no efficient (poynomial time) algorithm that succeeds in recovering the support of $\beta^*$ w.h.p. Evidence for this conjecture comes from the provable failure of several known efficient algorithms in this regime. Specifically in \cite{wainwright2009sharp} it is shown that LASSO provably fails to recover the support of $\beta^*$ w.h.p. when $n<(1-\epsilon)2n_{\mathrm{alg}}$ for any $\epsilon>0$, in the sense that for any $\beta^*$ the optimal solution of LASSO will not have the same signed support as $\beta^*$ w.h.p. Furthermore, via a combinatorial geometric argument  the authors in \cite{donoho2006counting} show that if $n<(1-\epsilon)2n_{\mathrm{alg}}$ for any $\epsilon>0$, then the optimal solutions of another estimator, similar to LASSO, called Basis Pursuit, also fails to recover the unknown the support of the unknown vector $\beta^*$ w.h.p. in the special case $\sigma^2=0$.

An attempt to explain the apparent algorithmic hardness in the general case when $n^* \leq n \leq cn_{\mathrm{alg}}$ is made in \cite{gamarnikzadik}, under the additional assumption that $\beta^*$ is exactly $k$-sparse, that is it has exactly $k$ non-zero coordinates, and binary (though the technique is expected to generalize from the binary case to the general case where $|\beta^*|_{\mathrm{min}} \geq 1$). The authors focus on the problem $$\begin{array}{clc}(\Phi_2)  & \min  &\|Y-X\beta\|_2 \\ &\text{s.t.}& \beta \in \{0,1\}^p,  \|\beta\|_0=k, 
\end{array}$$ and they prove that the optimal solution of this problem has approximately the same support as $\beta^*$ w.h.p., when $n>n^*$. Here and eslewhere $\|\beta\|_0$ is the number of non-zero coordinates of the vector $\beta$. Note that $\|\beta\|_0=k$ is not a convex constraint and thus $\Phi_2$ is not a priori an algorithmically tractable problem. The author study the geometry of the solutions space of $(\Phi_2)$ and show that when $n^*=2k \log p/\log \left(\frac{2k}{\sigma^2}+1\right)<n<ck \log p=cn_{\mathrm{alg}}$ for some sufficiently small $c>0$, a geometrical property called Overlap Gap Property (OGP) holds w.h.p. The OGP for this problem is the property that the exactly $k$-sparse $\beta$s that achieve near optimal cost for $\Phi_2$ split into two non-empty ``well-separated" categories; the ones whose support is close with the support of $\beta^*$ in the Hamming distance, and the ones whose support is far from the support of $\beta^*$ in the Hamming distance, creating a ``gap" for the vectors with supports in a ``intermediate" Hamming distance from the support of $\beta^*$. Similar forms of OGP are known in various random constraint satisfaction problems and statistical physics models such as the random $k$-SAT problem, proper coloring of a sparse random graph, the problem of finding a largest independent set of a random graph and many others. \cite{achlioptas2008algorithmic},\cite{AchlioptasCojaOghlanRicciTersenghi},\cite{montanari2011reconstruction},\cite{coja2011independent},\cite{gamarnik2014limits},\cite{gamarnik2014performance},\cite{rahman2014local},\cite{gamarnik2016finding}. For example, in a sparse random graph it has been proven that any two independent sets with size near optimality either have intersection at least of size $\tau_1>0$ or have intersection at most of size $\tau_2<\tau_1$, thus leading to a gap for the intermediate intersection sizes. The OGP for independent sets was used to establish fundamentals barriers for the so-called local algorithms for finding nearly largest independent sets in sparse random graphs \cite{gamarnik2014limits},\cite{gamarnik2014performance},\cite{rahman2014local}. Furthermore, it is a common feature of most of these problems that when the OGP ceases to hold, even very simple algorithms are able to succeed \cite{achlioptas2008algorithmic}. Motivated by these results the authors in \cite{gamarnikzadik} suggest the presence of OGP is the source of an algorithmic hardness for this high dimensional linear regression model in $n^* \leq n \leq c n_{\mathrm{alg}}$.


\section*{Results}
In this paper, we prove two sets of results supporting the conjecture that the OGP is the source of an algorithmic hardness in the regime $n^* \leq n \leq c n_{\mathrm{alg}}$.

(a) Our first set of results discusses the performance of the $\mathrm{LASSO}_\lambda$ for a wide range of tuning parameters $\lambda$.  We establish that if $n^* \leq n<cn_{\mathrm{alg}}$ for small enough $c>0$ and $\beta^*$ exactly $k$-sparse and binary, then for any $$\lambda \geq\sigma  \sqrt{\frac{1}{ k}} \exp \left(-\frac{k \log p}{5n}\right) $$ the optimal solution of $\mathrm{LASSO}_\lambda$ fails to $\ell_2$-stable recover $\beta^*$ w.h.p. Albeit our result does not apply for any arbitrarily small value of $\lambda>0$ our result covers certain arguably important choices of $\lambda$ in the literature of LASSO. More precisely, our results covers the theoretically successful choice of the tuning parameter $\lambda$ for LASSO when $n \geq Cn_{\mathrm{alg}}$ in \cite{bickel2009} which, as explained in the Introduction, shows that  $\mathrm{LASSO}_\lambda$ with $$\lambda=\lambda^*:=A \sigma \sqrt{ \log p/n}$$ for constant $A>2 \sqrt{2}$, $\ell_2$-stably recovers $\beta^*$  (see \cite[Chapter 11]{Book15}  for a simpler exposition). Indeed, since in our case $n<k \log p$ this choice of $\lambda$ satisfies trivially $$\lambda=\lambda^* \geq  A\sigma \sqrt{1 / k}>\sigma \sqrt{1 / k} $$ and therefore $\lambda=\lambda^* \geq \sigma  \sqrt{\frac{1}{ k}} \exp \left(-\frac{k \log p}{5n}\right) $.

An important feature of our result is that it is \textit{quantitative}, in the sense that it gives a lower bound of how far the optimal solution of $\mathrm{LASSO}_\lambda$ is from $\beta^*$ in the $\ell_2$ norm. In particular, we show that this lower bound depends exponentially on the ratio $k \log p/n$. Moreover, given the existing positive result of \cite{bickel2009} for LASSO, our result confirms that $n_{\mathrm{alg}}=k \log p$ is the \textit{exact order of necessary number of samples} for $\mathrm{LASSO}_\lambda$ to $\ell_2$-stably recover the ground truth vector $\beta^*$, when  $\lambda \geq \sigma \sqrt{1 / k} \exp \left(-k \log p/5n\right) $. Our result is therefore closed in spirit with the literature on LASSO for its performance for support recovery where the similar phase transition results are established by Wainwright in \cite{wainwright2009sharp}.

In the specific case $\beta^*$ is binary a natural modification of LASSO it is to add the box constraint $\beta \in [0,1]^p$ to the LASSO formulation. Such box constraints have been proven to improve the performance of LASSO in many cases, such as in signal processing applications \cite{Chris17}. We show that in our case,  our negative result for LASSO remains valid even with the box constraint. Specifically, let us focus for any $\lambda>0$ on \begin{align}\label{lassobox} \mathrm{LASSO(box)}_\lambda: \min_{\beta \in [0,1]^p} n^{-1}\|Y-X\beta\|^2_2+\lambda \|\beta\|_1.
\end{align}We show that if $n^* \leq n<cn_{\mathrm{alg}}$ for small enough $c>0$ and $\beta^*$ is an exactly $k$-sparse binary vector, for any $\lambda \geq \sigma  \sqrt{\frac{1}{ k}} \exp \left(-\frac{k \log p}{5n}\right)  $ the optimal solution of $\mathrm{LASSO(box)}_\lambda$ also fails to $\ell_2$-stably recover $\beta^*$ w.h.p.  

(b) Our second set of results concerns the Overlap Gap Property (OGP) and its implications. We first establish that if $n \geq Cn_{\mathrm{alg}}$ for some sufficiently large constant $C>0$, OGP indeed ceases to hold, proving the complementary part of a conjecture from the conference paper \cite{gamarnikzadik}. Furthermore, we prove that for these values of $n$ a very simple Local Search Algorithm exploits the ``smooth" geometrical structure of the solutions space which also leads to the absence of OGP and provably succeeds in both recovering both the support of $\beta^*$ and $\ell_2$ stable recovering vector of $\beta^*$. Notably this set of results applies for all sparsity levels $k \leq \frac{p}{3}$ and any $k$-sparse $\beta$ with $|\beta|_{\min} \geq 1$.

\subsection*{Beyond the Gaussian Assumption on $X$}

Our results on high-dimensional linear regression when $n \geq Cn_{\rm alg}$ are established under the idealized assumption on $X$ having iid $\mathcal{N}\left(0,1\right)$ entries.  Such an assumption allows a broader technical development and the establishment of tight statistical guarantees. Yet, naturally, the question is whether our structural results generalize beyond the present setting.

Regarding our first set of results on the performance of $\mathrm{LASSO}_\lambda$, the core technical tool used is the probabilistic result  \cite[Theorem 3.1]{gamarnikzadik}. The result is established for $X$ having iid $\mathcal{N}\left(0,1\right)$ entries, but is expected to a setting where $X$ has iid rows but with non-Gaussian and weakly dependent entries (see the Introduction of \cite{gamarnikzadik}).

Regarding our second set of results on the absence of OGP and analysis of the performance of LSA, we expect our results to generalize in a straightforward way to the case where $X$ with iid subGaussian entries (that is iid entries with bounded subGaussian norm), potentially tolerating on top of this weak dependence between the row entries. The reason is that our proof techniques are based on two key results; that Restricted Isometry Property (RIP) holds for the matrix $X$ and that the Hanson-Wright concentration inequality can be applied for quadratic forms defined by arbitrary matrix $A$ and random vectors of the form  $Xv$ for $v \in \mathbb{R}^{p}$ with $\|v\|_2=1$. Both these results are known to hold under the assumption of arbitrary iid subGaussian entries of $X$ and $W$ \cite{hanson1971, Baraniuk2008}. For this reason we consider our second set of result to generalize in a straightforward manner to the case of iid subGaussian entries. Furthermore, we consider the generalization to row entries that are weakly dependent potentially true for both RIP and Hanson-Wright inequality, yet we are not aware of such a result and such a pursuit would requires further technical work; we leave this as a interesting direction of future work. 
\section*{Notation}
 
For a matrix $A \in \mathbb{R}^{n \times n}$ we use its operator norm $\|A\|:= \max_{x \not = 0 }\frac{\|Ax\|_2}{\|x\|_2}$, and its Frobenius norm $\|A\|_F:=\left( \sum_{i,j} |a_{i,j}|^2 \right) ^{\frac{1}{2}}$. If $n,d \in \mathbb{N}$ and $A \in \mathbb{R}^{d \times p}$ by $A_i,i=1,2,\ldots,p$ we refer to the $p$ columns of $A$. For $p \in (0,\infty), d \in \mathbb{N}$ and a vector $x \in \mathbb{R}^d$ we use its $\mathcal{L}_p$-norm, $\|x\|_p:=\left( \sum_{i=1}^p |x_i|^p \right)^{\frac{1}{p}}$. For $p=\infty$ we use its infinity norm $\|x\|_{\infty} := \max_{i=1,\ldots,d} |x_i|$ and for $p=0$, its $0$-norm $\|x\|_0=|\{i \in \{1,2,\ldots, d\} | x_i \not =0 \}|$. We say that $x$ is $k$-sparse if $\|x\|_0 \leq k$ and exactly $k$-sparse if $\|x\|_0=k$.  We also define the support of $x$, $\mathrm{Support}\left(x\right):=\{i \in \{1,2,\ldots, d\} | x_i \not =0\}$. For $k \in \mathbb{Z}_{>0}$ we adopt the notation $[k]:=\{1,2,\ldots,k\}$. Finally with the real function $\log : \mathbb{R}_{>0} \rightarrow \mathbb{R} $ we refer everywhere to the natural logarithm.

 \section*{Structure of the Paper}
The remained of the paper is structured as follows.
The description of the model, assumptions and main results are found in the next section. Section 3 is devoted to the proof of the failure of the LASSO in the regime $n^* \leq n \leq c n_{\mathrm{alg}}$. Section 4 is devoted to the proof of the results related to the Overlap Gap Property and the success of the local search algorithm in the regime $n \geq C n_{\mathrm{alg}}$.

\section{Main Results and Proof Ideas}

We remind our model for convenience. Let $n,p,k \in \mathbb{N}_{>0}$ and $\sigma^2>0$. Let also $X \in \mathbb{R}^{n \times p}$ be an $n \times p$ matrix with i.i.d. $N(0,1)$ entries and $W \in \mathbb{R}^n$ be an $n \times 1$ vector with i.i.d. $N(0,\sigma^2)$ entries. We assume that $X,W$ are mutually independent. Let $\beta^*$ be a $p \times 1$ exactly $k$-sparse vector in $\mathbb{R}^p$, that is a $p$-dimensional vector with exactly $k$-non zero coordinates, and let $Y \in \mathbb{R}^n$ be an $n \times 1$ vector given by $Y=X \beta^*+W$. Assuming the knowledge of $(Y,X)$ and of the values of the parameters $n,p,k,\sigma^2$, we study the question of efficiently recovering the ground truth $\beta^*$ either by approximating its support or by $\ell_2$-stable recovering the vector itself or both.

 We are interested in the high dimensional regime where $p$, the number of features, exceeds $n$, the sample size, and both diverge to infinity. Various assumptions on $n,p,k,\sigma^2$ are required for technical reasons and some of the assumptions may vary from theorem to theorem, but they are always explicitly stated in the statements. Everywhere we assume that $k<n$, that is the number of samples is strictly larger than the sparsity level. The results hold in the ``with high probability" (w.h.p.) sense as $k,n,p$ diverge to infinity, but for concreteness we will usually explicitly say that $k$ diverges to infinity. This automatically implies the same for $p$ and $n$ since our assumptions always imply $k<n$ and clearly $k<p$.

\subsection{Below $n_{\mathrm{alg}}$ samples: Failure of the LASSO}

For this subsection we focus on the case where $\beta^*$ is exactly $k$-sparse and binary. In that case as stated in the introduction, if $k \leq \exp \left( C \sqrt{ \log p} \right)$ for some $C>0$, then $\beta^*$ can be exactly recovered with  order $n^*= 2k \log p/\log \left( \frac{2k}{\sigma^2}+1\right)$. In particular as $k/\sigma^2$ grows it implies an $\ell_2$-stable recovery guarantee with  $n \ll n_{\mathrm{alg}}$ samples. In this subsection we discuss the performance of LASSO in this regime. We show that when $n/n_{\mathrm{alg}}$ is sufficiently small, for a wide range of tuning parameters $\lambda$  $\mathrm{LASSO}_\lambda$, \textit{fails to $\ell_2$-stably recover} the ground truth vector $\beta^*$. Our result applies for $\mathrm{LASSO}_\lambda$ \textit{ with and without} box constraints.

Furthermore, our result applies for arbitrary choice of the tuning parameter $\lambda$ as long as \begin{align} \label{LambdaC}
\lambda \geq \frac{\sigma}{\sqrt{k}} \exp \left(-\frac{k \log p}{5n}\right).
\end{align} Note that this range of possible $\lambda$'s \textit{include the standard optimal choice in the literature} of the tuning parameter $\lambda=A\sigma \sqrt{   \log p/n}$ for constant $A>2\sqrt{2}$ \cite{bickel2009}, \cite[Chapter 11]{Book15} as in our regime we assume $n \leq k \log p$, hence for this choice of $\lambda$, it holds $\lambda \geq \sigma \sqrt{1/k}$ and in particular (\ref{LambdaC}) is satisfied.

We present now the result.

\begin{theorem}\label{LASSO}
Suppose that $\hat{C}\sigma^2 \leq k \leq \min\{1,\sigma^2\}\exp \left( C \sqrt{\log p} \right)$ for some constants $C,\hat{C}>0$. Then, there exists a constant $c>0$ such that the following holds. If $n^* \leq n \leq c n_{\mathrm{alg}}$, $\beta^* \in \mathbb{R}^p$ is an exactly $k$-sparse binary vector, arbitrary choice of $\lambda$ satisfying (\ref{LambdaC})  and $\hat{\beta}_{\mathrm{LASSO},\lambda}$, $\hat{\beta}_{\mathrm{LASSO}(\mathrm{box}),\lambda}$ are the optimal solutions of the formulations $\mathrm{LASSO}_{\lambda}$ and $\mathrm{LASSO(box)}_{\lambda}$ respectively, then

$$ \min \left(\|\hat{\beta}_{\mathrm{LASSO,\lambda}}-\beta^*\|_2,\|\hat{\beta}_{\mathrm{LASSO}(\mathrm{box}),\lambda}-\beta^*\|_2\right) \geq \exp\left(\frac{k \log p}{5n} \right) \sigma,$$

w.h.p. as $k \rightarrow +\infty$.

\end{theorem}

Note that $\ell_2$ stable recovery means finding a vector $\beta$ such that $\|\beta-\beta^*\|_2 \leq C'\sigma$ for some constant $C'>0$. The above theorem establishes that in the case of an exactly $k$-sparse and binary $\beta^*$, when the samples size is less than $k \log p$ both the optimal solutions of $\mathrm{LASSO}_{\lambda}$ and $\mathrm{LASSO(box)}_{\lambda}$ for any $\lambda$ satisfying (\ref{LambdaC}) fails to $\ell_2$-stable recover the ground truth vector $\beta^*$ by a multiplicative factor which is exponential on the ratio $\frac{k \log p}{n}$.  In particular, coupled with the result from \cite{bickel2009} this shows that $k \log p$ is the necessary and sufficient order of samples for which LASSO can $\ell_2$-stable recover $\beta^*$ for some $\lambda>0$ satisfying (\ref{LambdaC}).

\subsection{Above $n_{\mathrm{alg}}$ samples: The Absence of OGP and the success of the Local Search Algorithm}
In this setting we assume that $\beta$ is $k$-sparse, but not necessarily exaclty $k$-sparse. We establish the absence of the OGP in the  case $n \geq Cn_{\mathrm{alg}}=C k \log p$ for sufficiently large $C>0$, w.h.p. For the same values of $n$ we also propose a very simple Local Search Algorithm (LSA) for recovering the $k$-sparse $\beta^*$ which provably succeeds w.h.p. In fact our results for OGP is an easy consequence of the success of LSA.

\subsection*{The Absence of OGP}

We now state the definition of Overlap Gap Property (OGP) which generalizes the definition used in \cite{gamarnikzadik} where it focuses on the binary case for $\beta^*$.

\begin{definition}
Fix an instance of $X,W$. The regression problem defined by $(X,W,\beta^*)$ where a vector $\beta^*$ is an exactly $k$-sparse vector with $|\beta^*|_{\text{min}} \geq 1$ satisfies the Overlap Gap Property (OGP) if there exists $r=r_{n,p,k,\sigma^2}>0$ and constants $ 0<\zeta_1<\zeta_2<1$ such that

\begin{itemize}
\item[(1)] $\|Y-X\beta^*\|_2 <r$,

\item[(2)] There exists a $k$-sparse vector $\beta$ with $\mathrm{Support}(\beta) \cap \mathrm{Support}(\beta^*)=\emptyset$ and $\|Y-X\beta\|_2 <r$, and

\item[(3)] If a $k$-sparse vector $\beta$ satisfies $\|Y-X\beta\|_2<r$ then either 
$$|\mathrm{Support}(\beta) \cap \mathrm{Support}(\beta^*)|<\zeta_1k$$
or
$$|\mathrm{Support}(\beta) \cap \mathrm{Support}(\beta^*)|>\zeta_2k.$$
 
\end{itemize} 
\end{definition} The OGP has a natural interpretation. It states that the $k$-sparse $\beta$s which achieve near optimal cost for the objective value $\|Y-X\beta\|_2$ split into two non-empty ``well-separated" regions; the ones whose support is close with the support of $\beta^*$ in the Hamming distance sense, and the ones whose support is far from the support of $\beta^*$ in the Hamming distance sense, creating a ``gap" for the vectors with supports in a ``intermediate" Hamming distance.

In \cite{gamarnikzadik} the authors prove that under the assumption $\frac{1}{5}\sigma^2 \leq k \leq \min\{1,\sigma^2\}\exp \left( C \sqrt{\log p} \right)$ for some constant $C>0$ if $n$ satisfies $n^*<n \leq c k \log p$, for some sufficiently small constant $c>0$, then the OGP restricted for binary vectors holds for some $r>0$ and $\zeta_1=\frac{1}{5}$ and $\zeta_2=\frac{1}{4}$. Details can be found in the paper. As mentioned though in the introduction, it is conjectured in \cite{gamarnikzadik} that OGP will not hold when $n \geq C k \log p$ for some constant $C>0$, which is the regime for $n$ where efficient algorithms, such as LASSO, have been proven to work. We confirm this conjecture in the theorem below.

\begin{theorem}\label{OGP}
There exists $c,C>0$ such that if $\sigma^2 \leq c \min\{k,\frac{\log p}{\log \log p} \}$, $n \geq C n_{\mathrm{alg}}$ the following holds. If the $\beta^*$ is exactly $k$-sparse and satisfies $|\beta^*|_{\min} \geq 1$ then the regression problem $\left(X,W,\beta^*\right)$ does not satisfy the OGP w.h.p. as $k \rightarrow +\infty$.

\end{theorem}

We now give some intuition of how this result is derived. The proof is based on a lemma on the ``local" behavior of the $k$-sparse $\beta$s with respect to the optimization problem $$\begin{array}{clc}(\tilde{\Phi}_2)  & \min  &\|Y-X\beta\|_2 \\ &\text{s.t.}& \|\beta\|_0 \leq k.
\end{array}$$ We first give a natural definition of what a non-trivial local minimum is for $\tilde{\Phi}_2$.

\begin{definition}
We define a $k$-sparse $\beta$ to be a \textbf{non-trivial local minimum} for $\tilde{\Phi}_2$ if 
\begin{itemize} 
\item $\mathrm{Support}\left(\beta\right) \not = \mathrm{Support}\left(\beta^*\right)$, and 
\item if a $k$-sparse $\beta_1$ satisfies \begin{equation*} \max \{ |\mathrm{Support}\left(\beta\right) \setminus \mathrm{Support}\left(\beta_1\right)|, |\mathrm{Support}\left(\beta_1\right) \setminus \mathrm{Support}\left(\beta\right)|\} \leq 1, \end{equation*} it must also satisfy $$\|Y-X\beta_1\|_2 \geq \|Y-X\beta\|_2.$$
\end{itemize}
\end{definition}We continue with the observation that the presence of OGP deterministicaly implies the existence of a non-trivial local minimum for the problem $\tilde{\Phi}_2$.
\begin{proposition}\label{LocalMin}
Assume for some instance of $X,W$ the regression problem $(X,W,\beta^*)$ satisfies the Overlap Gap Property. Then for this instance of $X,W$ there exists at least one non-trivial local minimum for $\tilde{\Phi}_2$.
\end{proposition}
\begin{proof}
Assume that OGP holds for some values $r,\zeta_1,\zeta_2$. We choose $\beta_1$ the $k$-sparse vector $\beta$ that minimizes $\|Y-X\beta\|_2$ under the condition $|\mathrm{Support}(\beta) \cap \mathrm{Support}(\beta^*)| \leq \zeta_1 k$. The existence of $\beta_1$ is guaranteed as the space of $k$-sparse vectors with $|\mathrm{Support}(\beta) \cap \mathrm{Support}(\beta^*)|  \leq \zeta_1 k$ is closed under the Euclidean metric.

We claim this is a non-trivial local minimum. Notice that it suffices to prove that $\beta_1$ minimizes also $\|Y-X\beta\|_2$ under the more relaxed condition $|\mathrm{Support}(\beta) \cap \mathrm{Support}(\beta^*)| < \zeta_2 k$. Indeed then since $\zeta_1k<\zeta_2k$, $\beta_1$ will be the minimum over a region that contains its 2-neighborhood in the Hamming distance and as clearly the support of $\beta_1$ is not equal to the support of $\beta^*$ we would be done.

Now to prove the claim consider a $\beta$ with $\zeta_1k<|\mathrm{Support}(\beta) \cap \mathrm{Support}(\beta^*)|< \zeta_2 k$. By the Overlap Gap Property we know that it must hold $\|Y-X\beta\|_2>r$. Furthermore again by the Overlap Gap Property we know there is a $\beta'$ with $|\mathrm{Support}(\beta') \cap \mathrm{Support}(\beta^*)|=0 < \zeta_1 k$ for which it holds $\|Y-X\beta'\|_2<r$. But by the definition of $\beta_1$ it must also hold $\|Y-X\beta_1\|_2 \leq \|Y-X\beta'\|_2<r$ which combined with $\|Y-X\beta\|_2>r$ implies $\|Y-X\beta_1\|_2<\|Y-X\beta\|_2$. Since the $\beta$ was arbitrary with $\zeta_1k<|\mathrm{Support}(\beta) \cap \mathrm{Support}(\beta^*)|< \zeta_2 k$ the proof of the Proposition is complete.
\end{proof}

Now in light of the Proposition above, we know that a way to negate OGP is to prove the absence of non-trivial local minima for $\tilde{\Phi}_2$. We prove that indeed if $n \geq C k \log p$ for some universal $C>0$ our regression model does not have non-trivial local minima for $\tilde{\Phi}_2$ w.h.p. and in particular OGP does not hold in this regime w.h.p., as claimed. We state this as a separate result as it could be of independent interest.

\begin{theorem}\label{nolocal}
There exists $c,C>0$ such that if $\sigma^2 \leq c \min\{k,\frac{\log p}{\log \log p} \}$, $n \geq C n_{\mathrm{alg}}$ such that the following is true. If the $\beta^*$ is exactly $k$-sparse and satisfies $|\beta^*|_{\min} \geq 1$ then the optimization problem $(\tilde{\Phi}_2)$ has no non-trivial local minima w.h.p. as $k \rightarrow +\infty$.
\end{theorem}The complete proofs of both Theorem \ref{OGP} and Theorem \ref{nolocal} are presented in Section 4.

\subsection*{Success of Local Search}
 As stated in the introduction, in parallel to many results for random constrained satisfaction problems, the disappearance of OGP suggests the existence of a very simple algorithm succeeding in recovering $\beta^*$, usually exploiting the smooth local structure. Here, we present a result that reveals a similar picture. A natural implication of the absence of non-trivial local minima property is the success w.h.p. of the following very simple local search algorithm. Start with any vector $\beta_0$ which is $k$-sparse and then iteratively conduct ``local" minimization among all $\beta$'s with support of Hamming distance at most two away from the support of our current vector.

We now state this algorithm formally.
Let $e_i \in \mathbb{R}^p,i=1,2,\ldots,p$ be the standard basis vectors of $\mathbb{R}^p$.

\textbf{Local Search Algorithm (LSA)}

\begin{itemize}

\item[0.] Input: A $k$-sparse vector $\beta$ with support $S$.

\item[1.] For all $i \in S$ and $j  \in [p]$ compute $\mathrm{err}_{i}\left(j\right)=\min_{q} \|Y-X\beta+\beta_iX_i-qX_j\|_2$.

\item[2.] Find $(i_1,j_1)=\mathrm{argmin}_{i \in S,j \in [p]}\mathrm{err}_{i}\left(j\right) $ and $q_1:=\mathrm{argmin}_{q \in \mathbb{R}} \|Y-X\beta+\beta_{i_1}X_{i_1}-qX_{j_1}\|_2$.

\item[3.] If $\|Y-X\beta+\beta_{i_1}X_{i_1}-q_1X_{j_1}\|_2<\|Y-X\beta\|_2$, update the vector $\beta$ to $\beta-\beta_{i_1} e_{i_1}+qe_{j_1}$, the set $S$ to the support of the new $\beta$ and go to step 1. Otherwise terminate and output $\beta$.

\end{itemize} 
For the performance of the algorithm we establish the following result.

\begin{theorem}\label{LSA}
There exist $c,C>0$ so that if $\beta^* \in \mathbb{R}^p$ is an exactly $k$-sparse vector,  $n \geq C n_{\mathrm{alg}}$ and $\sigma^2 \leq  c |\beta^*|_{\text{min}}^2\min\{ \frac{\log p}{\log \log p},k\}$ then the algorithm LSA with an arbitrary $k$-sparse vector $\beta_0$ as input terminates in at most $ \frac{4k \|Y-X\beta_0\|_2^2 }{\sigma^2 n}$ iterations with a vector $\hat{\beta}$ such that \begin{itemize}
\item[(1)] $\mathrm{Support}\left( \hat{\beta} \right)=\mathrm{Support}\left( \beta^* \right)$ and
\item[(2)] $\|\hat{\beta}-\beta^*\|_2 \leq \sigma$,
\end{itemize}
w.h.p. as $k \rightarrow +\infty$.

\end{theorem}

The complete proof of Theorem \ref{LSA} requires some care and is approximately 16 pages long. It is fully presented in Section 4.

\section{Proof of Theorem \ref{LASSO}}

\subsection{Auxilary Lemmata}

\begin{lemma}\label{lem:simple}
Fix any $C_1>0$. Any vector $\beta$ that satisfies $\|\beta\|_1 \leq  k-C_1 \sigma \sqrt{k}$ also satisfies $\|\beta-\beta^*\|_2  \geq C_1 \sigma$.
\end{lemma}
\begin{proof}

Assume $\beta$ satisfies $\|\beta-\beta^*\|_2 \leq C_1 \sigma$.
We let $S$ denote the support of $\beta^*$, and let $\beta_S \in \mathbb{R}^p$ be the vector which equals to $\beta$ in the coordinates that correspond to $S$ and is zero otherwise.
We have by the triangle inequality and the Cauchy Schwartz inequality,
\begin{align*}
&k-\|\beta_S\|_1=\|\beta^*_S\|_1-\|\beta_S\|_1 \leq \| \beta_S-\beta^*_S \|_1 \leq \sqrt{k} \|\left(\beta-\beta^*\right)_S\|_2 \leq \sqrt{k} \|\beta-\beta^*\|_2  \leq C_1 \sigma \sqrt{k},
\end{align*} which gives $k-C_1 \sigma\sqrt{k} \leq \|\beta_S\|_1 \leq \|\beta\|_1$.
\end{proof}

We also need the Theorem 3.1. from \cite{gamarnikzadik}, which we re-state here for convenience.

\begin{theorem}[\cite{gamarnikzadik}]\label{theorem:PureNoise}

Let $Y' \in \mathbb{R}^n$ be a vector with i.i.d. normal entries with mean zero and abritrary variance $\mathrm{Var}(Y_1)$ and $X \in \mathbb{R}^{n \times p}$ be a matrix with iid standard Gaussian entries. Then for every $C>0$ there exists $c_0>0$ such that if $c<c_0$ and for some integer $k'$ it holds $k' \log k' \leq Cn$,
$k'\le \mathrm{Var}\left(Y'_1\right) \le 3k'$, and $ n\le ck'\log p$, then there exists an exactly $k'$-sparse binary $\beta$ such that 
\begin{align*}
n^{-{1\over 2}}\|Y-X\beta\|_2\le \exp\left(\frac{1}{2c} \right)\sqrt{k'+\mathrm{Var}\left(Y'_1\right)}\exp\left(-{k'\log p\over n}\right)
\end{align*}
 w.h.p. as $k'\rightarrow\infty$.
\end{theorem}

Finally, we establish the following Lemma.
\begin{lemma}\label{claimV}
Under the assumptions of Theorem \ref{LASSO} there exists universal constants $c>0$ such that the following holds. If $n^* \leq n \leq c k \log p$ then there exists $\alpha \in [0,1]^p$ with

 \begin{itemize}
\item[(1)] $n^{-\frac{1}{2}}\|Y-X\alpha\|_2\leq \sigma$
 
\item[(2)] $\|\alpha\|_1  = k-2C_1 \sigma \sqrt{k},$
\end{itemize} 
w.h.p. as $k \rightarrow + \infty$.
\end{lemma}
 \begin{proof}

Let $$\lambda:=1-4C_1\sqrt{ \frac{\sigma^2}{ k}}$$ and $$A_{C_1}=\{\lambda \beta^*+(1-\lambda)\beta | \beta \in \{0,1\}^p, \|\beta\|_0=k/2, \mathrm{Support}\left(\beta\right) \cap \mathrm{Support}\left(\beta^*\right)=\emptyset \}.$$ $A_{C_1}$ is the set of vectors of the form $\alpha:=\lambda \beta^*+(1-\lambda)\beta$ where $\beta$ is exactly $\frac{k}{2}$-sparse binary with support disjoint from the support of $\beta^*$. Since by our assumption $n>n^*$ or equivalently
 $$ \frac{k \log p}{5n}<\frac{1}{10}\log \left(1+\frac{2k}{\sigma^2} \right)$$ we conlude that for some $C'>0$ large enough, if $C' \sigma^2 \leq k$ then $$4C_1 \sqrt{\frac{\sigma^2}{ k}}=4\exp\left( \frac{k \log p}{5n}\right) \sqrt{\frac{\sigma^2}{ k}} < 4\left(1+\frac{2k}{\sigma^2} \right)^{\frac{1}{10}}\sqrt{\frac{\sigma^2}{ k}}<1.$$ In particular $\lambda >0$ and thus $\lambda \in [0,1]$. Therefore $A_{C_1} \subset [0,1]^p$. It is straightforward to see also that all these vectors have $\ell_1$ norm equal to $k \lambda+k(1-\lambda)/2=k( \lambda +1)/2$. But for our choice of $\lambda$ we have $$ k( \lambda +1)/2=k-2 C_1\sigma \sqrt{k}$$ Therefore for all $\alpha \in A_{C_1}$ it holds $\|\alpha\|_1 = k-2C_1 \sigma \sqrt{k}$ and $\alpha \in [0,1]^p$. In particular, in order to prove our claim it is enough to find $\alpha \in A_{C_1}$ with $n^{-\frac{1}{2}}\|Y-X\alpha\|_2\leq \sigma$.

We need to show that for some $c>0$, there exists w.h.p. a binary vector $\beta$ which is exactly $k/2$ sparse, has disjoint support with $\beta^*$ and also satisfies that \begin{equation*} n^{-\frac{1}{2}}\|Y-X(\lambda \beta^*+(1-\lambda)\beta)\|_2 \leq \sigma. \end{equation*} We notice the following equalities: 
\begin{align*}
\|Y-X(\lambda \beta^*+(1-\lambda)\beta)\|_2
&= \|X \beta^*+W - \lambda X \beta^*-\left(1-\lambda \right)X \beta \|_2\\
& =(1-\lambda) \|X \beta^*+\left(1-\lambda\right)^{-1}W-X \beta\|_2.
\end{align*}Hence the condition we need to satisfy can be written equivalently as
\begin{align*}
&n^{-\frac{1}{2}}\|X \beta^*+\left(1-\lambda\right)^{-1}W-X \beta\|_2 \leq \left(1-\lambda\right)^{-1}\sigma,
\end{align*}or equivalently \begin{align*}
n^{-\frac{1}{2}} \|Y'-X\beta\|_2 \leq \frac{1}{4}\sqrt{k} \exp \left( -\frac{k \log p}{5n} \right),
\end{align*} where for the last equivalence we set $Y':=X\beta^*+(1-\lambda)^{-1}W$ and used the definition of $\lambda$ for the right hand side. 

Now we apply Theorem \ref{theorem:PureNoise} for $Y'$ $X' \in \mathbb{R}^{n \times \left(p-k\right)}$, which is $X$ after we deleted the $k$ columns corresponding to the support of $\beta^*$, and $k'=k/2$. We first check that the assumptions of the Theorem are satisfied. For all $i$, $Y'_i$ are iid zero mean Gaussian with $$\mathrm{Var}\left(Y'_i\right)=k+\sigma^2\left(1-\lambda\right)^{-2}=k(1+\frac{1}{16}\exp \left( -\frac{2k \log p}{5 n}\right)).$$ In particular for some constant $c_0>0$ if $n \leq c_0 k \log p$ it holds $$k'=\frac{k}{2} \leq \mathrm{Var}\left(Y'_i\right) \leq 3k/2 =3k'.$$ Finally we need $k' \log k' \leq C'n$ for some $C'>0$. For $k'=\frac{k}{2}$ it holds $k' \log k' \leq k \log k$ and also as $\hat{C} \sigma^2 \leq k \leq \min\{1,\sigma^2\}\exp \left( C\sqrt{ \log p} \right)$ it can be easily checked that for some constant $C'>0$ it holds $k \log k \leq C'\frac{2k \log p}{\log \left( \frac{2k}{\sigma^2}+1\right)}=C'n^*$. As we assume $n \geq n^*$ we get $k' \log k' \leq C'n^* \leq Cn$ as needed. Therefore all the conditions are satisfied.

Applying Theorem \ref{theorem:PureNoise} we obtain that for some constant $c_1>0$ there exists w.h.p. an exactly $k/2$ sparse vector $\beta$ with disjoint support with $\beta^*$ and $$n^{-\frac{1}{2}}\|Y'-X\beta\|_2 \leq \exp \left( \frac{1}{2c_1} \right) \sqrt{k'+\mathrm{Var}\left(Y'_i\right)} \exp\left(-\frac{k' \log (p-k)}{n} \right).$$ Plugging in the value for $k'$ and using $\mathrm{Var}\left(Y'_i\right) \leq \frac{3}{2}k$ we conclude the w.h.p. existence of a binary $k/2$-sparse vector $\beta$ with disjoint support with $\beta^*$ and $$n^{-\frac{1}{2}}\|Y'-X\beta\|_2 \leq \exp \left( \frac{1}{2c_1} \right)\sqrt{2k} \exp\left( -\frac{k \log (p-k)}{2n}\right).$$ Finally we need to verify
$$\exp \left( \frac{1}{2c_1} \right) \sqrt{2k} \exp\left( -\frac{k \log (p-k)}{2n}\right) \leq \frac{1}{4}\sqrt{k} \exp \left( -\frac{k \log p}{5n} \right) .$$ We notice that as $k/\sqrt{p} \rightarrow 0$ as $k,p \rightarrow +\infty$, which is true since we assume $k \leq \exp \left(C \sqrt{\log p}\right),$ we have $$\exp \left( \frac{1}{2c_1} \right) \sqrt{2k} \exp\left( -\frac{k \log (p-k)}{2n}\right) \leq \exp \left( \frac{1}{2c_1} \right) \sqrt{2k} \exp\left( -\frac{k \log p}{3n}\right), \text{ for large enough } k,p.$$ Hence we need to show $$\exp \left( \frac{1}{2c_1} \right) \sqrt{2k} \exp\left( -\frac{k \log p}{3n}\right) \leq \frac{1}{4}\sqrt{k} \exp \left( -\frac{k \log p}{5n} \right) .$$ or equivalently 
$$ \exp \left( \frac{1}{2c_1} \right) \sqrt{2} \leq  \frac{1}{4} \exp \left(\frac{2k \log p}{15n} \right)$$ which is clearly satisfied if $n \leq c_3 k \log p$ for some constant $c_3>0$. Therefore choosing $c=\min\{c_1,c_3\}$ the proof of the claim and of the theorem is complete. 

\end{proof}

\subsection{Proofs of Theorem \ref{LASSO}}
In this subsection we use the Lemmata from the previous subsections and prove the Theorem \ref{LASSO}.

\begin{proof}[Proof of Theorem \ref{LASSO}]
Let
 \begin{align}\label{eq:C_1}
C_1:=\exp\left(\frac{k \log p}{5n}\right).
\end{align}

According the Lemma \ref{lem:simple} it suffices to show that for $C_1$ given by (\ref{eq:C_1}), \begin{align} \label{goalL}\max\{\|\beta_{\mathrm{LASSO},\lambda}\|_1, \|\beta_{\mathrm{LASSO(box)},\lambda}\|_1 \leq k-C_1 \sigma \sqrt{k},\end{align} w.h.p. as $k \rightarrow + \infty$. 

To show this, we notice that since $\beta_{\mathrm{LASSO},\lambda}$ and  $ \beta_{\mathrm{LASSO(box)},\lambda}$ are the optimal solution to $\mathrm{LASSO}_{\lambda}$ and $ \mathrm{LASSO(box)}_{\lambda}$ respectively, they obtains objective value smaller then any other feasible solution. Note that $\alpha$ given in Lemma \ref{claimV} is feasible for both quadratic optimization problems. Hence it holds almost surely, \begin{align} \label{firstLasso}
\max_{v \in \{\beta_{\mathrm{LASSO},\lambda},\beta_{\mathrm{LASSO(box)},\lambda}  \} }\{ \frac{1}{n} \|Y-Xv\|_2^2+\lambda_p \|v\|_1\} \leq   \frac{1}{n} \|Y-X\alpha\|_2^2+\lambda_p\|\alpha\|_1 
\end{align}
Hence we conclude that w.h.p. as $k \rightarrow + \infty$,
\begin{align}
\lambda_p \max\{\|\beta_{\mathrm{LASSO},\lambda}\|_1, \|\beta_{\mathrm{LASSO(box)},\lambda}\|_1  
& \leq  \max_{v \in \{\beta_{\mathrm{LASSO},\lambda}, \beta_{\mathrm{LASSO(box)},\lambda}  \} }\{ \frac{1}{n} \|Y-Xv\|_2^2+\lambda_p \|v\|_1\} \notag \\
&  \leq   \frac{1}{n} \|Y-X\alpha\|_2^2+\lambda_p\|\alpha\|_1 \text{ , using (\ref{firstLasso})} \notag\\
 & \leq  \sigma^2+\lambda_p \left(k-2C_1 \sqrt{k}\sigma\right) \text{, using Lemma \ref{claimV}} \notag
\end{align}or by rearranging,
\begin{align}\label{Las1}
\lambda_p\left(k-C_1 \sigma \sqrt{k}-\max\{\|\beta_{\mathrm{LASSO},\lambda}\|_1, \|\beta_{\mathrm{LASSO(box)},\lambda}\|_1 \right) \geq \left(\lambda_p C_1 \sqrt{k}-\sigma \right)\sigma.
\end{align} By assumption on $\lambda_p$ satisfying (\ref{LambdaC}) we conclude from (\ref{eq:C_1}) that $$\lambda_p C_1 \sqrt{k} \geq \sigma.$$ Combining the last inequality we have that the right hand side of (\ref{Las1}) is nonnegative, and therefore  (\ref{Las1}) implies that (\ref{goalL}) holds w.h.p. as $k \rightarrow +\infty$. This completes the proof of the Theorem \ref{LASSO}.
\end{proof}

\section{ LSA Algorithm and the Absence of the OGP}

\subsection{ Preliminaries}

   We introduce the notion of a super-support of a finite dimensional real vector.
\begin{definition}
Let $d \in \mathbb{N}$. We call a set $\emptyset \not = S \subseteq [d]$ a \textbf{super-support} of a vector $x \in \mathbb{R}^d$ if $\mathrm{Support}\left(x\right) \subseteq S$.
\end{definition}
  
 We also need the definition and some basic properties of the Restricted Isometry Property (RIP).
 
 \begin{definition}
Let $n,k,p \in \mathbb{N}$ with $k \leq p$. We say that a matrix $X \in \mathbb{R}^{n \times p}$ satisfies the \textbf{ $k$-Restricted Isometry Property ($k$-RIP)} with restricted isometric constant $\delta_{k} \in (0,1)$ if for every vector $\beta \in \mathbb{R}^p$ which is $k$-sparse it holds $$ (1-\delta_{k})\|\beta \|_2^2n \leq \|X \beta \|_2^2 \leq (1+\delta_{k}) \|\beta\|_2^2n.$$
\end{definition}

 A proof of the following theorem can be found in \cite{Baraniuk2008}.

\begin{theorem}\label{RIPthm}\cite{Baraniuk2008}
Let $n,k,p \in \mathbb{N}$ with $k \leq p$. Suppose $X \in \mathbb{R}^{n \times p }$ has i.i.d. standard Gaussian entries. Then for every $\delta>0$ there exists a constant $C=C_{\delta}>0$ such that if $n \geq Ck \log p$ then $X$ satisfies the $k$-RIP with restricted isometric constant $\delta_{k}<\delta$ w.h.p.
\end{theorem}

We need the following properties of RIP.

\begin{proposition}\label{properties}
Let $n,k,p \in \mathbb{N}$ with $k \leq p$. Suppose $X \in \mathbb{R}^{n \times p }$ satisfies the $k$-RIP with restricted isometric constant $\delta_{k} \in (0,1)$. Then for any $v,w \in \mathbb{R}^p$ which are $k$-sparse,
\begin{itemize}
\item[(1)] $$|(Xv)^{T} (Xw)| \leq  (1+\delta_{k})\|v\|_2\|w\|_2n \leq 2\|v\|_2\|w\|_2n.$$
\item[(2)] If $v,w$ have a common super-support of size $k$ then $$\|Xw\|_2^2+4\|v-w\|_2\|w\|_2n+2\|v-w\|_2^2n \geq \|Xv\|_2 \geq \|Xw\|_2^2-4\|v-w\|_2\|w\|_2n.$$
\item[(3)] If $v,w$ have disjoint supports and a common super-support of size $k$ then $$|(Xv)^{T} (Xw)| \leq   \delta_k\left(\|v\|_2^2+\|w\|_2^2 \right) n.$$
\end{itemize}
\end{proposition}

\begin{proof}
 The first part follows from the Cauchy-Schwarz inequality and the definiton of $k$-RIP applied to the vectors $v,w$.
 For the second part we write $Xv=X(w+(v-w))$, and we have $$\|Xv\|_2^2=\|Xw\|_2^2+2\left(X(v-w)\right)^T\left(Xw\right) +\|X(v-w)\|_2^2.$$ Since $v,w$ have a common super-support of size $k$, the vectors $v-w,w$ are $k$-sparse vectors. Hence from the first part we have \begin{equation*} -2\|v-w\|_2\|w\|_2n \leq |X(v-w)^TXw| \leq 2\|v-w\|_2\|w\|_2n \end{equation*}
 \begin{equation*}
 0 \leq \|X(v-w)\|_2^2 \leq 2\|v-w\|_2^2n.
 \end{equation*} Applying these inequalities to the last equality, the proof follows.

 For the third part since $v,w$ are $k$-sparse and have a common super-support of size $k$ the vectors $v+w$ and $v-w$ are $k$-sparse vectors. Hence by $k$-RIP and that $v,w$ have disjoint supports we obtain \begin{equation*}\|X(v+w)\|_2^2 \leq (1+\delta_{k})\|v+w\|_2^2n=(1+\delta_{k})\left(\|v\|_2^2+\|w\|_2^2\right)n\end{equation*} and similarly \begin{equation*}\|X(v-w)\|_2^2 \geq (1-\delta_{k})\left(\|v\|_2^2+\|w\|_2^2\right)n.\end{equation*} Hence
 \begin{align*}
 &|(Xv)^T(Xw)|=|\frac{1}{4}\left[\|X(v+w)\|_2^2-\|X(v-w)\|_2^2\right]|\\
 & \leq \frac{1}{4} |(1+\delta_{k})\left(\|v\|_2^2+\|w\|_2^2\right)n-(1-\delta_{k})\left(\|v\|_2^2+\|w\|_2^2\right)n |\\
 & \leq  \delta_k\left(\|v\|_2^2+\|w\|_2^2 \right) n,
 \end{align*}
 
 as required.
\end{proof}

Finally, we need the so-called Hanson-Wright inequality.  
 \begin{theorem}[Hanson-Wright inequality, \cite{hanson1971}]
 There exists a constant $d>0$ such that the following holds. Let $n \in \mathbb{N}, A \in \mathbb{R}^{n \times n}$ and $t \geq 0$. Then for a vector $X \in \mathbb{R}^n$ with i.i.d. standard Gaussian components

$$\mathbb{P}\left( 	| X^tAX-\E{X^tAX} |>t \right) \leq 2 \exp \left[-d \min \left( \frac{t^2}{\|A\|^2_{\text{F}}},\frac{t}{\|A\|}\right) \right].$$
\end{theorem}

\subsection{Key Propositions on the Local Structure of $(\tilde{\Phi}_2)$}

To establish the Theorems \ref{OGP}, \ref{nolocal} and \ref{LSA} we need to obtain certain structural results on the local minima of $(\tilde{\Phi}_2)$.

The central object of interest is what we name as a $\alpha$-deviating local minimum ($\alpha$-DLM).
\begin{definition}
Let $n,p \in \mathbb{N}$,$\alpha \in (0,1),X \in \mathbb{R}^{n \times p}$ and $\emptyset \not = S_1,S_2,S_3 \subseteq [p]$. A triplet of vectors $(a,b,c)$ with $a,b,c \in \mathbb{R}^p$ is called an $\alpha$-\textbf{deviating local minimum ($\alpha$-D.L.M.)} with respect to $S_1,S_2,S_3$ and to the matrix $X$ if the following are satisfied: \begin{itemize}
\item The sets $S_1,S_2,S_3$ are pairwise disjoint and the vectors $a,b,c$ have super-supports $S_1,S_2,S_3$ respectively.
\item $|S_1|=|S_2| $ and $|S_1|+|S_2|+|S_3| \leq 3k$.
\item For all $i \in S_1$ and $j \in S_2$
\begin{equation}\label{eq:dfndlm}
\|\left(Xa-a_iX_i\right)+\left(Xb-b_jX_j\right)+Xc\|_2^2  \geq \|Xa+Xb+Xc\|_2^2-\alpha\left(\frac{\|a\|_2^2}{|S_1|}+\frac{\|b\|_2^2}{|S_2|}\right)n. 
\end{equation}
\end{itemize}
\end{definition}

\begin{remark}
 In several cases in what follows we call a triplet $(a,b,c)$ an \textbf{$\alpha$-DLM} with respect to a matrix $X$ without explicitly referring to their corresponding super-sets $S_1,S_2,S_3$ but we do always assume their existence. 
\end{remark}

We establish several results on $\alpha$-DLMs. We start with the following algebraic claim for the DLM property.
\begin{claim}\label{equivalence}
Let $n,p,k \in \mathbb{N}$ with $k \leq \frac{1}{3}p$. Suppose a matrix $X \in \mathbb{R}^{n \times p}$ satisfies the $3k$-RIP for some isometric constant $\delta_{3k} \in (0,1)$ and that for some $\alpha \in (0,1)$ a triplet $(a,b,c)$ is an $\alpha$-D.L.M. with respect to $X$. Then
 $$\|X\left(a+b\right)\|_2^2+2(Xc)^T\left(X(a+b)\right) \leq \left(\alpha+4\delta_{3k}\right)\left(\|a\|_2^2+\|b\|_2^2\right)n. $$
 \end{claim}
\begin{proof}
Let $S_1,S_2,S_3$ the super-sets of the vectors $a,b,c$ with respect to which the triplet $(a,b,c)$ is an $\alpha$-DLM. Set $m:=|S_1|=|S_2|$. Based on the definition of an $\alpha$-DLM by expanding the squared norm in the left hand side of (\ref{eq:dfndlm}) we have that  $\forall i \in S_1, j \in S_2$ it holds $$a_i^2\|X_i\|_2^2+b_j^2\|X_j\|_2^2+2a_ib_jX_i^T X_j-2\left(Xa+Xb+Xc \right)^T \left(a_iX_i+b_jX_j \right) \geq -\alpha\left(\frac{\|a\|_2^2}{m}+\frac{\|b\|_2^2}{m}\right)n.$$Summing over all $i \in S_1,j \in S_2$ we obtain $$\sum_{i \in S_1,j \in S_2} \left[ a_i^2\|X_i\|_2^2+b_j^2\|X_j\|_2^2+2a_ib_jX_i^T X_j-2\left(Xa+Xb+Xc \right)^T  \left(a_iX_i+b_jX_j \right) \right] \geq -m\alpha\left(\|a\|_2^2+\|b\|_2^2\right)n$$
which equivalently gives
$$m \sum_{i \in S_1}  a_i^2\|X_i\|_2^2 +m \sum_{j \in S_2} b_j^2\|X_j\|_2^2 + 2(Xa)^T(Xb)-2m\left(Xa+Xb+Xc \right)^T \left(Xa+Xb \right) \geq -m\alpha\left(\|a\|_2^2+\|b\|_2^2\right)n$$ which after rearranging and multiplying with $-\frac{1}{m}$ implies that the quantity
\begin{align*}
&\|X\left(a+b\right)\|_2^2+2(Xc)^T\left(X(a+b)\right)+2\underbrace{\left(1-\frac{1}{m}\right)(Xa)^T(Xb)}_{S}\\
&+\underbrace{\left[\|Xa\|_2^2-\sum_{i \in S_1}a_i^2 \|X_i\|_2^2\right]+\left[\|Xb\|_2^2-\sum_{j \in S_2}b_j^2 \|X_j\|_2^2\right]}_{T}
\end{align*} is at most $\alpha\left(\|a\|_2^2+\|b\|_2^2\right)n$. To finish the proof it suffices to establish that $S,T$ are both bounded from below by $-2\delta_{3k}\left(\|a\|_2^2+\|b\|_2^2\right)n$. We start with bounding $S$. The vectors $a,b$ have disjoint supports which sizes sum up to at most $3k$. In particular, the union of their supports is a common super-support of them of size at most $3k$. Hence we can apply part (3) of Proposition \ref{properties} to get
\begin{align*}
S&=2\left(1-\frac{1}{m}\right) (Xa)^T(Xb) \geq -2\delta_{3k}\left(1-\frac{1}{m}\right)\left(\|a\|_2^2+\|b\|_2^2\right)n \geq -2\delta_{3k}\left(\|a\|_2^2+\|b\|_2^2\right)n.
\end{align*}

For $T$ it suffices to prove that $\left[\|Xa\|_2^2-\sum_{i \in S_1}a_i^2 \|X_i\|_2^2\right] \geq -2\delta_{3k}\|a\|_2^2n$ and since the same will hold for $b$ by symmetry, by summing the inequalities we will be done. Note that as $a$ and all the standard basis vectors are $3k$-sparse vectors by $3k$-RIP for $X$ we have $\|Xa\|_2^2 \geq (1-\delta_{3k})\|a\|_2^2n$ and secondly $\|X_i\|_2^2 \leq (1+\delta_{3k})n$, for all $i \in [p]$.
Combining we obtain
\begin{align*}
&\left[\|Xa\|_2^2-\sum_{i \in S_1}a_i^2 \|X_i\|_2^2\right] \geq \left[(1-\delta_{3k})\|a\|_2^2n-(1+\delta_{3k})\sum_{i \in S_1}a_i^2 n\right] = -2\delta_{3k}\|a\|_2^2n.
\end{align*}

The proof is complete.
\end{proof}

We now state two key properties for D.L.M. triplets. We present the proof of the first Proposition here, as it is rather short. We defer the proof of the second Proposition to Subection \ref{proofProp2} because of its length.

\begin{proposition}\label{prop1}
Let $n,p,k \in \mathbb{N}$ with $k \leq \frac{1}{3}p$. Suppose that $X \in \mathbb{R}^{n \times p}$ satisfies the $3k$-RIP with restricted isometric constant $\delta_{3k}<\frac{1}{12}$. Then there is no $\frac{1}{4}$-D.L.M. triplet $(a,b,c)$ with respect to the matrix $X$ with $\|a\|_2^2+\|b\|_2^2 \geq \frac{1}{4} \|c\|_2^2.$
\end{proposition}

\begin{proof}

By Lemma \ref{equivalence} any $\frac{1}{4}$-D.L.M. triplet satisfies $$\|X\left(a+b\right)\|_2^2+2(Xc)^T\left(X(a+b)\right) \leq \left(\frac{1}{4}+4\delta_{3k}\right)\left(\|a\|_2^2+\|b\|_2^2\right)n.$$ But using the $3k$-R.I.P. for $X$ and that $a,b,c$ have disjoint supports with sizes summing up to at most $3k$ we get the following two inequalities from Proposition (\ref{properties});
\begin{itemize}
\item $\|X(a+b)\|_2^2 \geq (1-\delta_{3k}) \left( \|a+b\|_2^2\right)n=(1-\delta_{3k}) \left( \|a\|_2^2+\|b\|_2^2\right)n$, since $a+b$ is $3k$-sparse and $a,b$ have disjoint supports.

\item $(Xc)^T(X(a+b)) \geq -\delta_{3k}\left(\|c\|_2^2+\|a\|_2^2+\|b\|_2^2 \right)$, from Proposition \ref{properties} (3).
\end{itemize}

We obtain 
$$(1-\delta_{3k})\left( \|a\|_2^2+\|b\|_2^2\right)-2\delta_{3k}\left(\|c\|_2^2+\|a\|_2^2+\|b\|_2^2 \right)$$
is at most $ \left( \frac{1}{4}+4\delta_{3k} \right) \left( \|a\|_2^2+\|b\|_2^2\right).$ But now, this inequality can be equivalently written as 
\begin{equation}\label{equiavelent2} \left( \frac{3}{4}-7 \delta_{3k} \right)\left(\|a\|_2^2+\|b\|_2^2\right) \leq  \delta_{3k} \|c\|_2^2 .
\end{equation}Now we use that for $\delta_{3k}< \frac{1}{12}$ it holds $\frac{3}{4}-7 \delta_{3k} > 2 \delta_{3k}$. Using this in (\ref{equiavelent2}) we conclude that $\sqrt{\|a\|_2^2+\|b\|_2^2} < \frac{1}{2} \|c\|_2$ and the proof of the proposition is complete. 

\end{proof}

The second property of D.L.M. triplets we want is the following.

\begin{proposition}\label{prop2}
Let $n,p,k \in \mathbb{N}$ with $k \leq \frac{1}{3}p$. Suppose $X \in \mathbb{R}^{n \times p}$ has i.i.d. $N(0,1)$ entries. There exists constants $c_1,C_1>0$ such that if $n \geq C_1k \log p$ then w.h.p. there is no $\frac{1}{4}$-D.L.M. triplet $(a,b,c)$ with respect to the some sets $\emptyset \not = S_1,S_2,S_3 \subset [p]$ and the matrix $X$ such that the following conditions are satisfied. \begin{itemize}
\item[(1)] $|a|_{\min}:=\min \{ |a_i| : a_i \not = 0\} \geq 1$.
\item[(2)] $S_1 \cup S_3=[k] \cup \{p\}$, $p \in S_3$ and $S_1=\mathrm{Support}(a)$.
\item[(3)] $ \|a\|_2^2+\|b\|_2^2+\|c\|_2^2 \leq c_1 \min \{\frac{\log p}{\log( \log p)},k\}$.
\end{itemize} 
\end{proposition}
The proof is deferred to Subection \ref{proofProp2}.

\subsection{Proof of Theorems \ref{OGP}, \ref{nolocal} and \ref{LSA}}
We first prove Theorem \ref{LSA} and then we show how it implies Theorems \ref{OGP} and \ref{nolocal}.

\begin{proof}[Proof of Theorem \ref{LSA}]

Let $X'$ be an $n \times \left( p+1\right)$ matrix such that for all $i \in [n], j \in [p]$ it holds $X'_{i,j}=X_{i,j}$ and for $i \in [n],j=p+1$, $X'_{i,p+1}:=\frac{1}{\sigma} W_i$. In words, we create $X'$ by augmenting $X$ with the rescaled $\frac{1}{\sigma}W$ as an extra column. Note that $X'$ has iid standard Gaussian entries and furthermore $Y=X\beta^*+W=X'\begin{bmatrix}
           \beta^* \\
           \sigma
         \end{bmatrix}.$

 Notice that the performance of our algorithm is invariant with respect to rescaling of the quantities $Y,\beta^*,\sigma,\beta_0$ by a scalar. In particular by rescaling $Y=X\beta^*+W$ with $\frac{1}{|\beta^*|_{\min}}$ we can replace $Y$ by $\frac{Y}{|\beta^*|_{\min}}$, $\beta^*$ with $\frac{\beta^*}{|\beta^*|_{\min}}$, $\sigma^2$ by $\frac{\sigma^2}{|\beta^*|^2_{\min}}$ and finally $\beta_0$ by $\frac{\beta_0}{|\beta^*|^2_{\min}}$ and thus we may assume for our proof that $|\beta^*|_{\min}=1$. Notice that in this case our desired upper bound on the running time remains $4 k\frac{\|Y-X\beta_0\|_2^2}{\sigma^2 n} $ and our assumptions on the variance of the noise is now simply $\sigma^2 \leq c \min\{\frac{\log p}{\log \log p},k\}$ for some $c>0$.

Recall that the desired output of the algorithm are vectors $\hat{\beta}$ satisfying the following termination conditions.\\

\textbf{Termination Conditions:}
\begin{itemize}

\item[(TC1)] $\mathrm{Support}\left( \hat{\beta} \right)=\mathrm{Support}\left( \beta^* \right)$ and,
\item[(TC2)] $\|\hat{\beta}-\beta^*\|_2 \leq \sigma$.
\end{itemize}

We start with the following deterministic claim.

\begin{claim}\label{begin} Assume that the algorithm LSA has the following property. For any $k$-sparse $\beta$ which violates at least one of (TC1),(TC2) we have $\|Y-X\beta'\|_2^2 \leq \|Y-X\beta\|_2^2-\frac{\sigma^2}{4k}n$, where $\beta'$ is obtained from $\beta$ in one iteration of the LSA. Then the algorithm LSA terminates for any $k$-sparse vector $\beta_0$ as input in at most  $4 k\frac{\|Y-X\beta_0\|_2^2}{\sigma^2n} $ iterations with an output vector $\beta$ satisfying both conditions $(TC1),(TC2)$.
\end{claim}
\begin{proof}

The property clearly implies that for the algorithm to terminate it needs to satisfy both conditions $(TC1),(TC2)$. Hence we need to bound only the termination time appropriately. But since at every iteration that the algorithm does not terminate the quantity $\|Y-X\beta\|_2^2$ decreases by at least $\frac{\sigma^2}{4k}n$, the result follows.

\end{proof}

For any vector $v \in \mathbb{R}^p$ and $\emptyset \not = A \subseteq [p]$ we denote by $v_A \in \mathbb{R}^p$ the $p$-dimensional real vector such that $(v_A)_i=v_i$ for $i \in A$ and $(v_A)_i=0$ for $i \not \in A$. Furthermore we set $v_{\emptyset}=0_{p \times 1}$ for any vector $v$. Without the loss of generality from now on we assume $\mathrm{Support}\left( \beta^* \right)=[k]$. Following the Claim \ref{begin} and our discussion, in order to prove Theorem \ref{LSA} it suffices to prove that there exists $c,C>0$ such that w.h.p. there is no $k$-sparse $\beta$ that violates at least one of (TC1),(TC2) and furthermore satisfies that $\|Y-X\beta'\|_2^2 \geq \|Y-X\beta\|_2^2-\frac{\sigma^2}{4k}n$, where $\beta'$ is obtained from $\beta$ in one iteration of the LSA.

Suppose the existence of such a $\beta$. We first choose $C>0$ large enough so that $X'$ satisfies the $3k$-RIP with $\delta_{3k}<\frac{1}{12}$. The existence of this $C>0$ is guaranteed by Theorem \ref{RIPthm}. Denote by $T$ a super support of $\beta$, that satisfies $|T|=k$ and $T \cap [k]= \mathrm{Support}\left(\beta\right) \cap [k]$. The existence of $T$ is guaranteed as $|\mathrm{Support}\left( \beta\right)| \leq k$ and $k \leq \frac{p}{3}$. Note that in particular (TC1) is satisfied if and only if $\mathrm{Support}\left(\beta\right) = [k]$ if and only if $T=[k]$. We know that for all $i \in [p]$, $j \in T$ and $q \in \mathbb{R}$, $$\|Y-X\beta+\beta_jX_j-q X_i\|_2^2 \geq \|Y-X\beta\|_2 -\frac{\sigma^2}{4k}n$$  or equivalently,

\begin{align}\label{main_assum}
\|X\beta^*+W-X\beta+\beta_jX_j-qX_i \|_2^2 \geq \|X \beta^*+W-X\beta\|_2 -\frac{\sigma^2}{4k}n, \forall i \in [p], j \in T, q \in \mathbb{R}.
\end{align}

Consider the triplets $(a,b,c),(d,e,g) \in \mathbb{R}^{p+1} \times\mathbb{R}^{p+1} \times\mathbb{R}^{p+1} $, where
\begin{align*}
&a:=\begin{bmatrix}
           \beta^*_{[k] \setminus T} \\
           0
         \end{bmatrix}, b:=\begin{bmatrix}
           -\beta_{T \setminus [k]} \\
           0
         \end{bmatrix}, c:=\begin{bmatrix}
           (\beta^*-\beta)_{[k] \cap T} \\
           \sigma
         \end{bmatrix}
\end{align*}
         and 
\begin{align*}
d:=\begin{bmatrix}
           (\beta^*-\beta)_{[k]\cap T} \\
           0
         \end{bmatrix}, f:=\begin{bmatrix}
           0_{p \times 1} \\
           0
         \end{bmatrix},g:=\begin{bmatrix}
           (\beta^*)_{[k] \setminus T}-(\beta)_{T \setminus [k]}\\
           \sigma
         \end{bmatrix}. 
\end{align*}
         
\begin{lemma}\label{simple}
Assume that $ \|(\beta-\beta^*)_{[k]\cap T}\|_2^2 \geq  \sigma^2$. Then the inequalities (\ref{main_assum}) imply that the triplet $(d,f,g)$ is $\frac{1}{4}$-DLM with respect to the matrix $X'$.
\end{lemma} 

\begin{proof}
We use the relation (\ref{main_assum}) and we choose   $i=j \in [k] \cap T$ , and $q=\beta^*_i$ to get that 
$$\|X \beta^*+W-X\beta+ (\beta_i -\beta^*_i)X_i\|_2^2 \geq \|X \beta^*+W-X\beta\|_2 -\frac{\sigma^2}{4k}n,\\ \text{ for all } i \in [k] \cap T.$$

But now notice that with respect to $X' \in \mathbb{R}^{n \times (p+1)}$ and the vectors $d,f,g$ defined above this condition can be written as  
\begin{align}\label{eq:X'1}
\|X'd+X'f+X'g-d_iX'_i\|_2^2 \geq \|X'\left(d+f+g\right)\|_2^2-\frac{\sigma^2}{4k}n,\text{ for all } i \in [k] \cap T.
\end{align}But based on our assumptions we have $$\frac{\|d\|_2^2+\|f\|_2^2}{|[k] \cap T|} = \frac{\|(\beta-\beta^*)_{[k]\cap T}\|_2^2}{|[k] \cap T|} \geq \frac{\sigma^2}{|[k] \cap T|} \geq \frac{\sigma^2}{k}$$ which combined with the inequality above gives,\begin{align}\label{eq:X'1}
\|\left(X'd-d_iX'_i\right)+X'f+X'g\|_2^2 \geq \|X'd+X'f+X'g\|_2^2-\frac{1}{4} \frac{\|d\|_2^2+\|f\|_2^2}{|[k] \cap T|}n,\text{ for all } i \in [k] \cap T,
\end{align} which by definition since $f=0$ says that $(d,f,g)$ is a $\frac{1}{4}$-DLM triplet with respect to $[k] \cap T$, $U$ and $\mathrm{Support}(g)$, where $U$ is an arbitrary set of cardinality $|[k]\cap T|$ which is disjoint from $[k] \cap T$ and $\mathrm{Support}(g)$.
\end{proof}

Recall that $\beta$ does not satisfy at least one of (TC1) and (TC2). We now consider different cases with respect to that.

\textbf{Case 1:} $T=[k]$ but $\|\beta-\beta^*\|_2^2 > \sigma^2$.

In that case $\|(\beta-\beta^*)_{[k] \cap T} \|_2^2 \geq \sigma^2$, because $T=[k]$. In particular, from Claim \ref{simple} we know that $(d,f,g)$ is a $\frac{1}{4}$-DLM triplet with respect to the matrix $X'$. From Lemma \ref{prop1} since we assume that $X'$ satisfies the $3k$-RIP with $\delta_{3k}<\frac{1}{12}$ w.h.p. we know that for $(d,f,g)$ to be a $\frac{1}{4}$-DLM triplet it needs to satisfy $$\|d\|_2^2+\|f\|_2^2 < \frac{1}{4} \|g\|_2^2, \text{ w.h.p.}$$  which equivalently means $$\|\left(\beta-\beta^*\right)_{[k] \cap T}\|_2^2  < \frac{1}{4} \left( \|\beta^*_{[k] \setminus T}\|_2^2+\|\beta_{T \setminus [k]}\|_2^2+\sigma^2 \right) \text{ w.h.p.}$$
or equivalently as $T=[k]$
$$ \|\beta-\beta^*\|_2^2 < \frac{\sigma^2}{4} \text{w.h.p.}$$
This is a contradiction with our assumption on $\beta$ that $\|\beta-\beta^*\|_2^2 > \sigma^2$. Therefore indeed this case leads w.h.p. to a contradiction and the proof in this case is complete.

\textbf{Case 2:} $T \not = [k]$.

We start by proving that in this case if we choose $c<1$ then the inequalities (\ref{main_assum}) imply deterministically that $(a,b,c)$ is an $\frac{1}{4}$-DLM triplet with respect to $[k] \setminus T$, $T \setminus [k]$ and $\left([k] \cap T \right) \cup \{p+1\}$ and the matrix $X'$. For $i \in [k] \setminus T$, $j \in T \setminus [k]$ and $q = \beta^*_j$  (\ref{main_assum}) implies
 $$\|X \beta^*+W-X\beta+ \beta_j X_j -\beta^*_i X_i\|_2^2 \geq \|X \beta^*+W-X\beta\|_2 -\frac{\sigma^2}{4k}n,\\ \text{ for all } i \in [k] \setminus T,j \in T \setminus [k].$$ But now notice that with respect to $X' \in \mathbb{R}^{n \times (p+1)}$, and the vectors $a,b,c$ defined above, this condition can be written as  
\begin{align}\label{eq:X'1}
\|X'a+X'b+X'c-a_iX'_i-b_jX'_j\|_2^2 \geq \|X'\left(a+b+c\right)\|_2^2-\frac{\sigma^2 n}{4k},\\ \text{ for all } i \in [k] \setminus T,j \in T \setminus [k]
\end{align}Furthermore since the non-zero elements of $a$ are non-zero elements of $\beta^*$ we know $|a|_{\min} \geq 1$. In particular for all $i \in [k] \setminus T$ it holds $a_i^2 \geq 1$ and therefore for $m=|[k] \setminus T|$ it holds $\frac{\|a\|_2^2+\|b\|_2^2}{m} \geq |a|_{\min} \geq 1$. Therefore the inequality above implies 
  \begin{align}\label{eq:X'1}
\|X'a+X'b+X'c-a_iX'_i-b_jX'_j\|_2^2 \geq \|X'a+X'b+X'c\|_2^2-\frac{\sigma^2n}{4k} \left(\frac{\|a\|_2^2+\|b\|_2^2}{m}\right),\\  \text{ for all } i \in [k] \setminus T,j \in T \setminus [k]
\end{align}Finally, since we are assuming $c<1$ we have $\sigma^2 \leq k$ and therefore \begin{align}
\|\left(X'a-a_iX'_i\right)+\left(X'b-b_jX'_j\right)+X'c\|_2^2 \geq \|X'a+X'b+X'c\|_2^2-\frac{n}{4} \left(\frac{\|a\|_2^2+\|b\|_2^2}{m}\right),\\ \text{ for all } i \in [k] \setminus T,j \in T \setminus [k]
\end{align} 
which since $m=k-|[k] \cap T|=|[k] \setminus T|=|T \setminus [k]|$ is exactly the property that $(a,b,c)$ is $\frac{1}{4}$-DLM with respect to the sets $[k] \setminus T$, $T \setminus [k]$ and $\left( [k] \cap T \right) \cap \{p+1\}$ and the matrix $X'$. Since we assume that $X'$ satisfies the $3k$-RIP with $\delta_{3k} <\frac{1}{12}$ we conclude from Proposition \ref{prop1} that $\|a\|_2^2+\|b\|_2^2 \leq \frac{1}{4}\|c\|_2^2$  or equivalently,
 \begin{align}\label{eq:sigma1} 
 \|\beta^*_{[k] \setminus T}\|_2^2+\|\beta_{T \setminus [k]}\|_2^2 \leq \frac{1}{4} \left( \|(\beta-\beta^*)_{[k] \cap T}\|_2^2+\sigma^2 \right).
\end{align}

Now we apply Proposition \ref{prop2} for the $\frac{1}{4}$-DLM triplet $(a,b,c)$ with respect to $S_1:=[k] \setminus T$, $S_2:=T \setminus [k]$ and $S_3:=\left([k] \cap T \right) \cup \{p+1\}$. Let $c_1,C_1>0$ the corresponding constants of the proposition. We choose our $C$ to satisfy $C>C_1$ so that the hypothesis of the Proposition \ref{prop2} applies for any $\frac{1}{4}$-DLM triplet with respect to our matrix $X'$. In particular since $(a,b,c)$ is a $\frac{1}{4}$-DLM triplet we know that it should not satisfy one of the conditions w.h.p.  We have that $|a|_{\min} \geq 1$ and it is easy to check that $S_1 \cup S_3=[k] \cup \{p+1\}$, $p+1 \in S_3$ and $S_1=\mathrm{Support}(a)$. Therefore from the conclusion of Proposition \ref{prop2} it must be true that the triplet $(a,b,c)$ must violate the third condition, that is
\begin{align*}
&c_1 \min\{ \frac{\log p}{\log \log p},k\} \leq \|a\|_2^2+\|b\|_2^2+\|c\|_2^2, \text{ w.h.p.}
\end{align*}
or equivalently
\begin{align*}
c_1 \min\{\frac{\log p}{\log \log p},k\} \leq \|\beta^*_{[k] \setminus T}\|_2^2+\|\beta_{T \setminus [k]}\|_2^2+  \|(\beta-\beta^*)_{T\cap [k]}\|_2^2+\sigma^2, 
\end{align*}Applying inequality (\ref{eq:sigma1}) with the last inequality we conclude
\begin{align*}
&c_1 \min\{\frac{\log p}{\log \log p},k\} \leq \frac{1}{4} ( \|(\beta-\beta^*)_{[k] \cap T}\|_2^2+\sigma^2)+  \|(\beta-\beta^*)_{T\cap [k]}\|_2^2+\sigma^2,
\end{align*}
or equivalently
\begin{align*}
\frac{4}{5}c_1 \min\{\frac{\log p}{\log \log p},k\}-\sigma^2 \leq  \|(\beta-\beta^*)_{[k] \cap T}\|_2^2, 
\end{align*}
Choosing our constant $c>0$ to satisfy $c < \frac{2}{5}c_1$, we can assume $2\sigma^2<\frac{4}{5}c_1\min\{ \frac{\log p}{\log \log p},k\}$ and therefore the last inequality implies
\begin{align}\label{surma}
\sigma^2 \leq \|\left(\beta-\beta^*\right)_{[k] \cap T}\|_2^2,
 \end{align}

This by Lemma \ref{simple} implies that $(d,f,g)$ is also an $\frac{1}{4}$-DLM triplet. In particular from Proposition \ref{prop1} we have  $$\|d\|_2^2+\|f\|_2^2 < \frac{1}{4} \|g\|_2^2,$$ which equivalently means  $$\|\left(\beta-\beta^*\right)_{[k] \cap T}\|_2^2  < \frac{1}{4} \left( \|\beta^*_{[k] \setminus T}\|_2^2+\|\beta_{T \setminus [k]}\|_2^2+\sigma^2 \right).$$ 
Using (\ref{eq:sigma1}) the above inequality implies w.h.p.
$$\|\left(\beta-\beta^*\right)_{[k] \cap T}\|_2^2   < \frac{1}{4} \left(1/4 ( \|(\beta-\beta^*)_{[k] \cap T}\|_2^2+\sigma^2)+\sigma^2 \right)$$ which implies $$\|\left(\beta-\beta^*\right)_{[k] \cap T}\|_2^2 < \frac{1}{3} \sigma^2,$$ a contradiction with the inequality (\ref{surma}).

\end{proof}

\begin{proof}[Proof of Theorem \ref{OGP} and Theorem \ref{nolocal}]
Given Proposition \ref{LocalMin} we only need to establish Theorem \ref{nolocal} to establish both of the Theorems, that is we only need to prove that there is no non-trivial local minimum for $(\tilde{\Phi}_2)$ w.h.p.
We choose constants $c,C>0$ so that the conclusion of Theorem \ref{LSA} is valid.  Suppose the existence of a $k$-sparse vector $\beta$ which is a non-trivial local minimum for $(\tilde{\Phi}_2)$, that is it satisfies the following conditions (a),(b);

\begin{itemize} 
\item[(a)] $\mathrm{Support}\left(\beta\right) \not = \mathrm{Support}\left(\beta^*\right)$, and 
\item[(b)] if a $k$-sparse $\beta_1$ satisfies \begin{equation*} \max \{ |\mathrm{Support}\left(\beta\right) \setminus \mathrm{Support}\left(\beta_1\right)|, |\mathrm{Support}\left(\beta_1\right) \setminus \mathrm{Support}\left(\beta\right)|\} \leq 1, \end{equation*} it must also satisfy $$\|Y-X\beta_1\|_2 \geq \|Y-X\beta\|_2.$$
\end{itemize}

We feed now $\beta$ as an input for the algorithm (LSA). From condition (b) we know that the algorithm will terminate immediately without updating the vector. But from Theorem \ref{LSA} we know that the output of LSA with arbitrary $k$-sparse vector as input will output a vector satisfying conditions $(1),(2)$ of Theorem \ref{LSA} w.h.p. In particular, since $\beta$ was the output of LSA with input itself, it should satisfy condition (1) w.h.p., that is $\mathrm{Support}\left(\beta\right) = \mathrm{Support}\left(\beta^*\right)$, w.h.p. which contradicts the definition of $\beta$ (condition (a)). Therefore w.h.p. there does not exist a non-trivial local minimum for $(\tilde{\Phi}_2)$. This completes the proof.
\end{proof}

\subsection{Proof of Proposition \ref{prop2}}\label{proofProp2}
Here we present the deferred proof of Proposition \ref{prop2}. 
\begin{proof}[Proof of Proposition \ref{prop2}]

We first choose $C_1>0$ large enough based on Theorem \ref{RIPthm} so that $n \geq C_1 k \log p$ implies that $X$ satisfies the $3k$-RIP with $\delta_{3k}<\frac{1}{16}$ w.h.p.  In particular all the probability calculations below will be conditioned on this high-probability event.

We start with a lemma for bounding the probability that a specific triplet $(a,b,c)$ is an $\frac{1}{2}$-D.L.M. triplet with respect to $X$.
\begin{lemma}\label{Hanson_app}
There exists a $c_0>0$ such that for any fixed triplet $(a,b,c)$  with $a \not =0$, $$\mathbb{P}\left( (a,b,c) \text{ is a } \frac{1}{2}\text{-D.L.M. triplet} \right) \leq 2\exp\left(-c_0 n\min\{1,\frac{\|a\|_2^2+\|b\|_2^2}{\|c\|_2^2}\}\right),$$
where for the case $c=0$ we abuse the notation by defining $\frac{1}{0}:=+\infty$.
\end{lemma}

\begin{proof}
We prove only the case $c \not = 0$. The case $c  =0$ is similar. Assume a fixed triplet $(a,b,c)$ is an $\frac{1}{2}$-DLM. Using Claim \ref{equivalence} we have that it holds $$\|X\left(a+b\right)\|_2^2+2(Xc)^T\left(X(a+b)\right) \leq \left(\frac{1}{2}+4\delta_{3k} \right)\left(\|a\|_2^2+\|b\|_2^2\right)n.$$ We set $X_1=X\left(\frac{a+b}{\sqrt{\|a\|_2^2+\|b\|_2^2}}\right)$ and $W_1=X\left(\frac{c}{\|c\|_2}\right)$ and notice that  $X_1,W_1$ have independent $N(0,1)$ entries because $a,b,c$ have disjoint supports. The last inequality can be expressed with respect to $X_1,W_1$ as, \begin{align*}
&\|X_1\|_2^2+2\frac{\|c\|_2}{\sqrt{\|a\|_2^2+\|b\|_2^2}}W_1X_1 \leq \left( \frac{1}{2}+4\delta_{3k}\right) n.
\end{align*}

Now we introduce matrix notation. For $I_n$ the $n \times n$ identity matrix we set 
$$A:= 
  \left[ {\begin{array}{cc}
    I_n  &\frac{\|c\|_2}{\sqrt{\|a\|_2^2+\|b\|_2^2}}I_n \\
   \frac{\|c\|_2}{\sqrt{\|a\|_2^2+\|b\|_2^2}} I_n & 0_n\\
  \end{array} } \right]$$ and $V$ be the $2n$ vector obtained by concatenating $X_1,W_1$, that is  $V:=(X_1,W_1)^t$.  Then the last inequality can be rewritten with respect to the matrix notation as $$V^tAV \leq \left( \frac{1}{2}+4\delta_{3k}\right) n.$$ 
    We now bound the probability of this inequality. First note that since $V$ is a vector with iid standard Gaussian elements it holds that $\E{V^tAV}=\mathrm{trace}\left(A\right)=n.$ Hence,
  \begin{align*}
&\mathbb{P}\left(V^tAV \leq \left( \frac{1}{2}+4\delta_{3k} \right) n \right)\\
   & \leq \mathbb{P}\left(|V^tAV-\E{V^tAV}| \geq (\frac{1}{2}-4\delta_{3k})n \right), \text{ using } \E{V^tAV}=n,\\
   & \leq \mathbb{P}\left(|V^tAV-\E{V^tAV}| \geq \frac{n}{4} \right), \text{ using that } \delta_{3k}<\frac{1}{16} \text{ implies }   \frac{1}{2}-4\delta_{3k} > \frac{1}{4} . 
\end{align*} Now we apply Hanson-Wright inequality, so we need to estimate the Frobenious norm and the spectral norm of the matrix $A$. We have  
  \begin{equation} \label{Frob}
\|A\|_F^2 \leq 3n \|A\|_{\infty}^2 \leq 3 \max\{1,\frac{\|c\|_2^2}{\|a\|_2^2+\|b\|_2^2}\}n.
\end{equation} Now using that $A$
can be represented as the Kronecker product $$ A=  \left[ {\begin{array}{cc}
   1  &\frac{\|c\|_2}{\sqrt{\|a\|_2^2+\|b\|_2^2}} \\
   \frac{\|c\|_2}{\sqrt{\|a\|_2^2+\|b\|_2^2}}  & 0\\
  \end{array} } \right] \otimes I_n$$ we obtain that the maximal eigenavalue of $A$ is the maximal eigenvalue of the $2\times 2$ first product term of the Kronecker product. In particular from this it can be easily checked that,\begin{equation}\label{spectral}
  \|A\| \leq 2 \max\{1,\sqrt{\frac{\|c\|_2^2}{\|a\|_2^2+\|b\|_2^2} }\}.
\end{equation}   
  
Now  from Hanson-Wright inequality we have for some constant $d>0$, \begin{equation} \label{final1}
\mathbb{P}\left( |V^tAV-\E{V^tAV} |\geq \frac{1}{4}n \right) \leq  2 \exp \left[-d \min \left( \frac{\frac{1}{16}n^2}{\|A\|^2_{\text{F}}},\frac{\frac{1}{4} n}{\|A\|}\right) \right]
\end{equation} Using (\ref{Frob}), (\ref{spectral}) and noticing that $\max\{1,\sqrt{\frac{\|c\|_2^2}{\|a\|_2^2+\|b\|_2^2} }\} \leq \max\{1,\frac{\|c\|_2^2}{\|a\|_2^2+\|b\|_2^2} \}$ we obtain that for the constant $c_0:=\frac{1}{48}d$ it holds $$d\min \left( \frac{\frac{1}{16}n^2}{\|A\|^2_{\text{F}}},\frac{\frac{1}{4} n}{\|A\|}\right) \geq c_0 n \min\{1,\frac{\|a\|_2^2+\|b\|_2^2}{\|c\|_2^2}\}$$ and therefore using (\ref{final1}) the proof is complete in this case.
\end{proof}

Now we proceed with the proof of the proposition. We define the following sets parametrized by $r,\tilde{c}>0$ and $\alpha \in (0,1)$ $$B_{r,\tilde{c}}:=\{ (a,b,c) \big{|} a,b,c \in \mathbb{R}^p,  \|a\|_0+\|b\|_0+\|c\|_0 \leq 2k+1, \|a\|_2^2+\|b\|_2^2+\|c\|_2^2 \leq r^2, |a|_{\min} \geq \tilde{c} \}$$ and\begin{align*}
&D_{\alpha,r,\tilde{c}} \text{ equal to}\\
&\{(a,b,c) \in B_{r,\tilde{c}} \big{|} (a,b,c) \text{ is } \alpha\text{-D.L.M. with correspondning super-supports satisfying}\\
& \text{the assumption (2) of the Proposition \ref{prop2} } \}
\end{align*}We call a triplet of sets $\emptyset \not = S_1,S_2,S_3 \subseteq [p]$ \textit{good} if \begin{itemize}
\item $S_1,S_2,S_3$ are pair-wise disjoint
\item $|S_1|=|S_2|$, $p \in S_3$ and $S_1 \cup S_3=[k] \cup \{p\}$
\end{itemize} For $\alpha \in \mathbb{R}$ and $S \subseteq \mathbb{R}$ we define the set $$S-\alpha:=\{s-\alpha|s \in S\}.$$ For $i=1,2,3$ we set $P_i:=\{(i-1)p+1,(i-1)p+2,\ldots,ip\}$. Notice that the sets $P_1,P_2,P_3$ partition $[3p]$.
We define the following family of subsets of $[3p]$, $$\mathcal{T}:=\{T \subset [3p] | \text{ the triplet } T \cap P_1,T\cap P_2-p,T \cap P_3-2p \text{ is good}\}.$$ It is easy to see that $\mathcal{T} \subset \{T \subset [3p] | |T| \leq 2k+1\}$. Furthermore for any $T \in \mathcal{T}$ we define

$$B_{r,\tilde{c}}(T):=\{ (a,b,c) \in B_{r,\tilde{c}}\big{|}  \mathrm{Support}\left( (a,b,c) \right) \subseteq T, T \cap P_1=\mathrm{Support}\left(a\right)\}$$ and \begin{align*}
&D_{\alpha,r,\tilde{c}}(T) \text{ equal to}\\
&\{ (a,b,c) \in B_{r,\tilde{c}}(T) \big{|} (a,b,c) \text{ is } \alpha\text{-D.L.M. with respect to } T \cap P_1,T\cap P_2-p,T \cap P_3-2p   \}.
\end{align*}We claim that \begin{align}\label{eq:Tunion}
D_{\frac{1}{4},r,1}=\bigcup_{T \in \mathcal{T}} D_{\frac{1}{4},r,1} \left(T\right).
\end{align}  For the one direction, if $A=(a,b,c) \in D_{\frac{1}{4},r,1} \left(T\right)$ for some $T \in \mathcal{T}$ then $(a,b,c)$ is $\alpha$-DLM with corresponding super-supports $T \cap P_1,T\cap P_2-p,T \cap P_3-2p$ which can be easily checked that they satisfy assumption (2) of the Proposition \ref{prop2} based on our assumptions. For the other direction if $A\in D_{\frac{1}{4},r,1}$ is an $\alpha$-DLM with respect to $S_1,S_2,S_3$ satisfying the assumption (2) of the Proposition, it can be easily verified that for the set $T=S_1 \cup \left(S_2+p \right) \cup \left(S_3+2p\right)$ it holds $T \in \mathcal{T}$ and furthermore $A \in D_{\frac{1}{4},r,1} \left(T\right)$.

Now to prove the proposition it suffices to prove that there exists $c_1,C_1>0$ such that if $n \geq C_1k \log p$ and $r=  \sqrt{c_1 \min \{ \frac{\log p}{\log \log p},k\}}$ then $$\lim_{k \rightarrow +\infty} \mathbb{P}\left(D_{\frac{1}{4},r,1} \not = \emptyset \right) =0.$$ Using the equation (\ref{eq:Tunion}) for $\alpha=\frac{1}{4}$ and $\tilde{c}=1$ and the union bound it suffices to be shown that for some $c_1,C_1>0$ if $n \geq C_1 k \log p$ and $r= \sqrt{c_1\min\{ \frac{\log p}{\log \log p},k\}}$ then $$\lim_{k \rightarrow +\infty}\sum_{T \in \mathcal{T}}   \mathbb{P}\left(D_{\frac{1}{4},r,1} \left(T\right) \not = \emptyset \right)=0.$$

We now state and prove the following packing lemma.
\begin{lemma}\label{packing}
There exists $C_2>0$ such that for any $r>0,\delta \in (0,1)$ and $T \in \mathcal{T}$ we can find $Q_{r,1-\delta}(T) \subseteq B_{r,1-\delta}(T)$ with the following two properties \begin{itemize}
\item $|Q_{r,1-\delta}(T)| \leq C_2\left(\frac{12r}{\delta}\right)^{2k+1}$. 
\item For any $p \in B_{r,1}(T)$ there exists $q \in Q_{r,1-\delta}(T)$ with $\|p-q\|_2 \leq \delta$. 
\end{itemize} 
\end{lemma}
\begin{proof}
Fix $r>0,\delta \in (0,1)$ and $T \in \mathcal{T}$. Since $T \subset [3p]$ and $|T| \leq 2k+1$ using standard packing arguments (see for example \cite{Baraniuk2008}) there exists universal constant $C_2>0$ and a set $$Q'_{r,1-\delta}(T) \subset B_{r}(T) :=\{ (a,b,c) \big{|} a,b,c \in \mathbb{R}^p,  \mathrm{Support}\left((a,b,c)\right) \subseteq T, \|a\|_2^2+\|b\|_2^2+\|c\|_2^2 \leq r^2 \}$$ with the properties that $|Q'_{r,1-\delta}(T)| \leq C_2\left(\frac{12r}{\delta}\right)^{2k+1}$ and that for any $p \in  B_r(T)$ there exists $q \in Q'_{r,1-\delta}(T)$ with $\|p-q\|_2 \leq \delta$.

To complete the proof we define $$Q_{r,1-\delta}\left(T\right)= Q'_{r,1-\delta}(T) \cap B_{r,1-\delta}(T).$$As $Q_{r,1-\delta}(T) \subseteq Q'_{r,1-\delta}(T)$ it also holds $$|Q_{r,1-\delta}(T)| \leq|Q'_{r,1-\delta}(T)| \leq C_2\left(\frac{12r}{\delta}\right)^{2k+1}.$$ For the other property let $p=(a,b,c) \in B_{r,1}(T)$. Since $B_{r,1}(T) \subseteq B_{r}(T)$ there exist $q=(l,m,n) \in Q'_{r,1-\delta}(T)$ with $\|p-q\|_2 \leq \delta$. We claim that $q \in B_{r,1-\delta}(T)$ which completes the proof. It suffices to establish $|l|_{\min} \geq 1-\delta$ and that $\mathrm{Support}\left(l\right) =T \cap P_1$.  We know $\|a-l\|_{\infty} \leq \|a-l\|_2 \leq \|p-q\|_2 \leq \delta.$ Therefore since for al $i \in T \cap P_1$, $|a_i|  \geq 1$ we get that for all $i \in T \cap P_1$, $|l_i| \geq 1-\delta$. Since $T \cap P_1$ was assumed to be a super-support of $l$ this implies both $\mathrm{Support}\left(l\right)=T \cap P_1$ and $|l|_{\min} \geq 1-\delta$. 

\end{proof}

\begin{claim}
Consider the sets $\{Q_{r,1-\delta}(T)\}_{T \in \mathcal{T}}$ from Lemma (\ref{packing}) defined for some $r>0$ and $0<\delta  \leq  \min\{\frac{1}{50r},\frac{1}{5}\}$. If $X$ satisfies the $3k$-RIP with $\delta_{3k} \in (0,1)$ then for any $T \in \mathcal{T}$ such that $D_{\frac{1}{4},r,1}(T) \not = \emptyset$, we have $Q_{r,1-\delta}(T) \cap D_{\frac{1}{2},r,\frac{1}{2}}(T) \not = \emptyset.$

\end{claim}

\begin{proof}
 
 To prove the claim, we consider an element $A=(a,b,c) \in D_{\frac{1}{4},r,1}(T)$. Note that since $A \in  D_{\frac{1}{4},r,1}(T) \subseteq B_{r,1}(T) \subset B_{r,1-\delta}(T)$ the definition of $Q_{r,1-\delta}(T)$ implies that for some $L=(l,m,g) \in Q_{r,1-\delta}(T)$ it holds $\|A-L\|_2 \leq \delta$. To complete the proof we show that $L \in D_{\frac{1}{2},r,\frac{1}{2}}(T)$.

Notice that from the definition of the sets $Q_{r,1-\delta}(T),D_{\frac{1}{4},r,1}(T)$, the vectors $a,l$ share the set $S_1=T \cap P_1$ as a common super-support and furthermore the vectors $b,m$ share the set $S_2=T\cap P_2$ as a common super-support. Since $A \in D_{\frac{1}{4},r,1}(T)$ we know firstly $S_1=\mathrm{Support}(a)$, secondly for any $i \in S_1=\mathrm{Support}(a)$, $|a_i| \geq 1$ and finally that for any $i \in S_1$ and $j \in S_2$ 
\begin{equation}\label{eq: Data_for_A} 
\|\left(Xa-a_iX_i+Xb-b_jX_j\right)+Xc\|_2^2  \geq \|X(a+b+c)\|_2^2-\frac{1}{4}\left(\frac{\|a\|_2^2}{|S_1|}+\frac{\|b\|_2^2}{|S_2|}\right)n. 
\end{equation}

To prove $L \in D_{\frac{1}{2},r,\frac{1}{2}}(T)$ it suffices to prove now firstly that $S_1=\mathrm{Support}(l)$, secondly for any $i \in \mathrm{Support}(l)$, $|l_i| \geq \frac{1}{2}$ and finally that for every $i \in S_1$ and $j \in S_2$ 
\begin{equation}  \label{eq:targer_for_L}
\|\left(Xl-l_iX_i+Xm-m_jX_j\right)+Xg\|_2^2  \geq \|X(l+m+g)\|_2^2-\frac{1}{2}\left(\frac{\|l\|_2^2}{|S_1|}+\frac{\|m\|_2^2}{|S_2|}\right)n.
\end{equation}We start with the first two properties. This is a similar calculation as in the proof of Lemma \ref{packing}. We know $\|a-l\|_2 \leq \|A-L\|_2 \leq \delta<\frac{1}{2}$. In particular, $\|a-l\|_{\infty} \leq \frac{1}{2}$. But we know that $S_1=\mathrm{Support}(a)$ and $|a|_{\min} \geq 1$. These together imply that for all $i \in S_1$, $|l_i| \geq \frac{1}{2}$. Since $S_1$ is a super-support of $l$ we conclude that indeed $S_1=\mathrm{Support}(l)$ and that for any $i \in \mathrm{Support}(l)$, $|l_i| \geq \frac{1}{2}$ as required. 
 Now we prove the third property and use Proposition \ref{properties}. By part (2) of this proposition we know that since $X$ satisfies the $3k$-RIP for some restricted isometric constant $\delta_{3k}<1$, any two vectors $v,w$ which share a common super-support of size at most $3k$ satisfy

 \begin{align} \label{eq:RIP_impl.}
\|Xw\|_2^2+4\|v-w\|_2\|w\|_2n+2\|v-w\|_2^2n \geq \|Xv\|_2 \geq \|Xw\|_2^2-4\|v-w\|_2\|w\|_2n
\end{align} For our convenience for the calculations that follow we set for all $i \in S_1$ and $j \in S_2$, $A_{i,j}:=A-a_ie_i-b_je_j$ and $L_{i,j}:=L-l_ie_i-m_je_j$, where by $\{e_i\}_{i \in [3p]}$ we denote the standard basis vectors of $\mathbb{R}^{3p}$. In words for all $i \in S_1$ and $j \in S_2$ we set $A_{i,j}$ the vector $A$ after we set zero its $i$ and $j$ coordinates and similarly we define $L_{i,j}$. Now fix $i \in S_1$, $j \in S_2$. Then we have by directly applying (\ref{eq:RIP_impl.}) for the two pairs $v=L_{i,j}$ and $w=A_{i,j}$ and $v=L,w=A$ that
$$\|X(A_{i,j})\|_2^2 \leq  \|X(L_{i,j})\|_2^2+4\|L_{i,j}-A_{i,j}\|_2\|A_{i,j}\|_2n+2\|L_{i,j}-A_{i,j}\|_2^2n$$ and 
$$\|X(A)\|_2^2 \geq \|X(L)\|_2^2-4\|A-L\|_2\|L\|_2n,$$Hence $\|X(A_{i,j})\|_2^2 -\|X(A)\|_2^2$ is at most
$$ \|X(L_{i,j})\|_2^2+4\|L_{i,j}-A_{i,j}\|_2\|A_{i,j}\|_2n+2\|L_{i,j}-A_{i,j}\|_2^2n-\|X(L)\|_2^2+4\|A-L\|_2\|A\|_2n.$$
But using the easy observations $$\|A_{i,j}-L_{i,j}\|_2 \leq \|A-L\|_2 \leq \delta$$ and $$\|A_{i,j}\|_2 \leq \|A\|_2 \leq r$$ we get that the last quantity can be upper bounded by $ \|XL_{i,j}\|_2^2-\|XL\|_2^2+(8\delta r+2\delta^2)n$. Therefore combining the last steps we have established $$\|X(A_{i,j})\|_2^2 -\|X(A)\|_2^2 \leq \|XL_{i,j}\|_2^2-\|XL\|_2^2+(8\delta r+2\delta^2)n.$$
 
But we know that by our assumptions
 $\|X(A_{i,j})\|_2^2 -\|X(A)\|_2^2 \geq -\frac{1}{4}\left(\frac{\|a\|_2^2}{|S_1|}+\frac{\|b\|_2^2}{|S_2|}\right)n. $ Therefore $$\|XL_{i,j}\|_2^2-\|XL\|_2^2 \geq -\frac{1}{4}\left(\frac{\|a\|_2^2}{|S_1|}+\frac{\|b\|_2^2}{|S_2|}\right)n-(8\delta r+2\delta^2)n.$$ So to prove (\ref{eq:targer_for_L}) it suffices to be proven that  \begin{align} \label{eq:final}
-\frac{1}{4}\left(\frac{\|a\|_2^2}{|S_1|}+\frac{\|b\|_2^2}{|S_2|}\right)n-(8\delta r+2\delta^2)n \geq -\frac{1}{2}\left(\frac{\|l\|_2^2}{|S_1|}+\frac{\|m\|_2^2}{|S_2|}\right)n.
\end{align} Note that  $\|A\|_2 \leq r,\|L\|_2 \leq r,\|A-L\|_2 \leq \delta$ implies $\|a\|_2^2-\|l\|_2^2 \leq 2 \delta r$ and $\|b\|_2^2-\|m\|_2^2 \leq 2 \delta r$. Hence from the definition of $A,L$ and since $|S_1|=|S_2| \geq 1$ it holds, $$\frac{1}{2}\left(\frac{\|a\|_2^2}{|S_1|}+\frac{\|b\|_2^2}{|S_2|}\right)n-\frac{1}{2}\left(\frac{\|l\|_2^2}{|S_1|}+\frac{\|m\|_2^2}{|S_2|}\right)n \leq 2 \delta r n.$$ In particular it holds $$-\frac{1}{2}\left(\frac{\|a\|_2^2}{|S_1|}+\frac{\|b\|_2^2}{|S_2|}\right)n \geq -\frac{1}{2}\left(\frac{\|l\|_2^2}{|S_1|}+\frac{\|m\|_2^2}{|S_2|}\right)n -2 \delta r n.$$ Hence using the last inequality we can immediately derive (\ref{eq:final}) provided that $$ \frac{1}{4}\left(\frac{\|a\|_2^2}{|S_1|}+\frac{\|b\|_2^2}{|S_2|}\right)n \geq 2\delta r n+(8\delta r+2\delta^2)n=(10\delta r+2\delta^2)n .$$ But now since $a_i^2 \geq 1$ for all $i \in S_1$, $\frac{\|a\|_2^2}{|S_1|} \geq 1$ and therefore $$ \frac{1}{4}\left(\frac{\|a\|_2^2}{|S_1|}+\frac{\|b\|_2^2}{|S_2|}\right)n \geq \frac{1}{4}n.$$ so it suffices that $2\delta^2+10 \delta r  \leq \frac{1}{4}$. It can be easily checked to be true if $\delta  \leq \min\{\frac{1}{50r},\frac{1}{5}\}$. The proof of the claim is complete.
\end{proof}

To prove the proposition we need to show that for some $c_1,C_1>0$ if $n \geq C_1 k \log p$, $r= \sqrt{c_1 \min\{ \frac{\log p}{\log \log p},k\}}$ and $\delta =\frac{1}{60r}$ then for the appropriately defined sets $\{Q_{r,1-\delta}(T)\}_{T \in \mathcal{T}}$ it holds
 $$ \lim_{k \rightarrow + \infty} \sum_{T \in \mathcal{T}} \mathbb{P}\left(|Q_{r,1-\delta}(T) \cap D_{\frac{1}{2},r,\frac{1}{2}}(T)| \geq 1\right) = 0.$$ But by Markov inequality for all such $T \in \mathcal{T}$, $$\mathbb{P}\left(|Q_{r,1-\delta}(T) \cap D_{\frac{1}{2},r,\frac{1}{2}}| \geq 1\right) \leq \E{|Q_{r,1-\delta}(T) \cap D_{\frac{1}{2},r,\frac{1}{2}}|}.$$ Furthermore for all $T \in \mathcal{T}$, $1 \leq |T \cap P_2| \leq k$. By the Markov inequality and summing over the possible values of $|T\cap P_2|$ for $T \in \mathcal{T}$, it suffices to show that for some $c_1,C_1>0$ if $n \geq C_1 k \log p$ and $r= \sqrt{c_1 \min\{  \frac{\log p}{\log \log p},k\}}$ then,

\begin{align}\label{eq:Target}
\lim_{k \rightarrow + \infty} \sum_{m=1}^{k} \sum_{T \in \mathcal{T}, |T \cap P_2|=m} \mathbb{E}\left(|Q_{r,1-\delta}(T) \cap D_{\frac{1}{2},r,\frac{1}{2}}(T)| \right) = 0
\end{align}

Fix $m \in [k]$ and a set $T \in \mathcal{T}$ with $|T\cap P_2|=m$. Then for any $A=(a,b,c) \in Q_{r,1-\delta}(T) \cap D_{\frac{1}{2},r,\frac{1}{2}}(T)$, since $D_{\frac{1}{2},r,\frac{1}{2}}(T) \subseteq B_{r,\frac{1}{2}}(T)$, we have $|a|_{\min} \geq \frac{1}{2}$ and $\|a\|_2^2+\|b\|_2^2+\|c\|_2^2 \leq r^2$. Based on the definition of $D_{\frac{1}{2},r,\frac{1}{2}}(T)$, we also have $|\mathrm{Support}(a)|=|S_1|=|S_2|=|T \cap P_2|=m$. Hence, $\|a\|_2^2 \geq |a|^2_{\min}m \geq \frac{1}{4}m$ and  $\|c\|_2^2 \leq \|a\|_2^2+\|b\|_2^2+\|c\|_2^2 \leq r^2$. By Lemma \ref{Hanson_app} we know that for any triplet $A=(a,b,c)$, $\mathbb{P}\left( A \in D_{\frac{1}{2},r,\frac{1}{2}}(T) \right) \leq  \exp \left(-c_0n \min \{1,\frac{\|a\|_2^2+\|b\|_2^2}{\|c\|_2^2} \} \right)$. Hence using the above inequalities we can conclude that for any such $A=(a,b,c) \in Q_{r,1-\delta}(T) $ it holds 
\begin{align}\label{eq:app1}
\mathbb{P}\left( A \in D_{\frac{1}{2},r,\frac{1}{2}}(T) \right) \leq 2\exp \left(-\frac{1}{4}c_0n \min \{1,\frac{m}{r^2} \} \right)
\end{align} Linearity of expectation, the above bound and the cardinality assumption on $Q_{r,1-\delta}(T)$ imply

\begin{align}\label{eq:app2}
& \E{|Q_{r,1-\delta}(T) \cap D_{\frac{1}{2},r,\frac{1}{2}}(T)|} \leq 2|Q_{r,1-\delta}(T)| \exp \left(-\frac{1}{4}c_0n \min \{1,\frac{m}{r^2} \} \right)\\
& \leq 2C_2 \left(\frac{12r}{\delta}\right)^{2k+1}\exp \left(-\frac{1}{4}c_0n \min \{1,\frac{m}{r^2} \} \right). 
\end{align}

We now count the number of possible $T \in \mathcal{T}$ with $ |T\cap P_2|=m$. Recall that any $T \subseteq [3p]$ satisfies $T \in \mathcal{T}$ if and only if the triplet of sets $T \cap P_1,T\cap P_2-p,T \cap P_3-2p$ is a \textit{good} triplet. That is if and only if \begin{itemize}
\item[(1)] $T \cap P_1,T\cap P_2-p,T \cap P_3-2p$ are pairwise disjoint sets and $|T \cap P_1|=|T \cap P_2-p|=|T \cap P_2|=m$
\item[(2)] $p \in T \cap P_3-2p $ 
\item[(3)] $\left( T\cap P_1 \right) \cup \left( T \cap P_3-2p \right)=[k] \cup \{p\}$
\end{itemize}
Since a set $T \subseteq [3p]$ is completely characterized by the intersections with $P_1,P_2,P_3$, it suffices to count the number of triplets of sets $T \cap P_i$, $i=1,2,3$ satisfying the three above conditions. Now conditions (1),(3) imply that $T\cap P_3$ is completely characterized by $T \cap P_1$. Furthermore by checking conditions $(1),(2),(3)$ we know that $T \cap P_1$ is an arbitrary subset of $[k]$ of size $m$. Hence we have $\binom{k}{m}$ choices for both the sets $T \cap P_1$ and $T \cap P_3$. Finally for the set $T \cap P_2$ we only have that it needs to satisfy $|T \cap P_2|=m$. Hence for $T \cap P_2$ we have $\binom{p}{m}$ choices, giving in total that the number of sets $T \in \mathcal{T}$ with $ |T \cap P_2|=m$  equals to $\binom{k}{m} \binom{p}{m}$. Hence,
$$ \sum_{T \in \mathcal{T}, |T\cap P_2|=m} \mathbb{E}\left(|Q_{r,1-\delta}(T)\cap D_{\frac{1}{2}}(T)| \right) \leq 2\binom{k}{m}\binom{p}{m} C_2 \left(\frac{12r}{\delta}\right)^{2k+1}\exp \left(-\frac{1}{4}c_0n \min \{1,\frac{m}{r^2} \} \right).$$ Summing over all $m=1,2,\ldots,k$  and using the bounds $\binom{k}{m} \leq 2^k,\binom{p}{m} \leq p^m$ we conclude that
\begin{align*}
&\sum_{m=1}^k\sum_{T \in \mathcal{T}, |T\cap P|=m} \mathbb{E}\left(|Q_{r,1-\delta}(T) \cap D_{\frac{1}{2},r,\frac{1}{2}}(T)| \right)
\end{align*}
is at most  
$$2C_3 k2^k \max_{m=1,\ldots,k} \left[ p^m \left(\frac{12r}{\delta}\right)^{2k+1}\exp \left(-\frac{1}{4}c_0n \min \{1,\frac{m}{r^2} \} \right)\right].$$

Therefore it suffices to show that for some $c_1,C_1>0$ if $n \geq C_1 k \log p$, $r=  \sqrt{c_1 \min\{ \frac{\log p}{\log \log p},k\}}$ and $\delta=\frac{1}{60r}$ then $$ \lim_{k \rightarrow  \infty} k 2^{k}\max_{m=1,\ldots,k} \left[ p^m \left(\frac{12r}{\delta}\right)^{2k+1}\exp \left(-\frac{1}{4}c_0n \min \{1,\frac{m}{r^2} \} \right)\right] = 0.$$Since this is an increasing quantity in $n$ and in $\frac{1}{\delta}$ we plug in $n=\frac{4}{c_0}C_1 k \log p$ and $\delta=\frac{1}{60 r}$ (since $r \rightarrow + \infty$) and after taking logarithms it suffices to be proven that for $C_1$ large enough but constant and $c_1>0$ small enough but constant, if $r= \sqrt{c_1 \min\{ \frac{\log p}{\log \log p},k\}}$ then $$\max_{m=1,\ldots,k} \left[ m \log p+ (2k+1)\log \left(1000 r^2\right)-C_1k \log p \min \{1,\frac{m}{r^2} \} \right] +k \log 2+ \log k \rightarrow -\infty.$$

We consider the two cases: when $m \leq  r^2$ and when $m \geq r^2$. Suppose $m \geq  r^2$, that is $\min \{1,\frac{m}{r^2} \}=1$. We choose $c_1$ small enough so that $1000r^2 \leq k \leq p$ and therefore 
\begin{align*}
&\max_{k \geq m \geq  r^2} \left[m \log p+ (2k+1)\log \left(1000 r^2\right)-C_1k \log p \min \{1,\frac{m}{r^2} \} \right]+k \log 2+ \log k\\
&=\max_{k \geq m \geq  r^2} \left[m \log p+ (2k+1)\log \left(1000 r^2\right)-C_1k \log p\right] +k \log 2+ \log k \\
& \leq -(C_1-4)k \log p + k \log 2+ \log k, \text{ since } m \log p+ (2k+1)\log \left(1000 r^2\right) \leq 4k \log p, \\
& \leq -(C_1-5) k \log p,
\end{align*} which if $C_1>6$ clearly diverges to $-\infty$ as $k \rightarrow +\infty$.

Now suppose $m \leq  r^2$, that is when $\min \{1,\frac{m}{r^2} \}=\frac{m}{r^2}$. We have 
\begin{align*}
&\max_{1 \leq m \leq  r^2} \left[m \log p+ (2k+1)\log \left(1000 r^2\right)-C_1k \log p \min \{1,\frac{m}{r^2} \} \right]+k \log 2+ \log k\\
&=\max_{1 \leq m \leq  r^2} \left[m \log p+ (2k+1)\log \left(1000 r^2\right)-C_1k \log p \frac{m}{r^2}\right] + k \log 2+ \log k. 
\end{align*} We write 
\begin{align*}
&m \log p+ (2k+1)\log \left(1000 r^2\right)-C_1k \log p \frac{m}{r^2}\\
&=m \log p -\frac{C_1}{2}k \log p \cdot \frac{m}{r^2}+(2k+1)\log \left(1000 r^2\right)- \frac{C_1}{2}k \log p \cdot  \frac{m}{r^2}.
\end{align*} 
But now for $c_1<1$ we have $r^2 \leq k$ and therefore
\begin{align}\label{eq:step1}
m \log p -\frac{C_1}{2}k \log p \cdot \frac{1}{4}\frac{m}{r^2} \leq (1-\frac{C_1}{2})m \log p \leq - 2\log p 
\end{align} for $C_1 \geq 6.$ Now we will bound the second summand.
Again assuming $C_1>6$ and using that $m \geq 1$ we have 
 \begin{align}\label{eq:step2}
(2k+1)\log \left(1000 r^2\right)- \frac{C_1}{2}k \log p \cdot  \frac{m}{r^2} \leq 3k \left(\log \left(1000 r^2\right)-\frac{1}{4r^2} \log p  \right)
\end{align} Now we claim that the right hand side of the above inequalty is at most $-3k$, given $c_1$ small enough, as $k \rightarrow + \infty$. It suffices to prove that if $r \leq  \sqrt{c_1 \frac{\log p}{\log \log p}}$ for some $c_1>0$ small enough then $\log \left(1000 r^2\right)-\frac{1}{4r^2} \log p \leq -1$ or equivalently $r^2\log \left(1000 r^2\right)+r^2 \leq \frac{1}{4}\log p$. But notice that the left hand side of the last inequality is increasing in $r$ and it can be easily checked that if $r^2=\frac{1}{5}\frac{\log p}{\log \log p}$ then $\frac{r^2\log \left(1000 r^2\right)+r^2 }{\log p}$ tends in the limit (as $p$ grows to infinity) to $\frac{1}{5}$ which is less than $\frac{1}{4}$. Therefore if $c_1< \frac{1}{5}$ the inequality becomes true for large enough $p$ for this value of $r$ and my monotonicity for all smaller values of $r$ as well. Now combining (\ref{eq:step1}) and (\ref{eq:step2}) we conclude that for small enough $c_1>0$ and large enough $C_1>0$ that 
 \begin{align*}
&\max_{1 \leq m \leq 4 r^2} \left[m \log p+ (2k+1)\log \left(1000 r^2\right)-C_1k \log p \frac{1}{4}\frac{m}{r^2}\right] +k \log 2+ \log k \\
& \leq -2 \log p-3k+k \log 2+ \log k\\
& \leq -(3-2 \log 2)k+\log k \rightarrow -\infty, \text{ as } n,p,k \rightarrow +\infty
\end{align*} which completes the proof.

\end{proof}

\bibliographystyle{plain}
\bibliography{bibliography}

\begin{thebibliography}{10}

\bibitem{AchlioptasCojaOghlanRicciTersenghi}
D.~Achlioptas, A.~Coja-Oghlan, and F.~Ricci-Tersenghi.
\newblock On the solution space geometry of random formulas.
\newblock {\em Random Structures and Algorithms}, 38:251--268, 2011.

\bibitem{achlioptas2008algorithmic}
Dimitris Achlioptas and Amin Coja-Oghlan.
\newblock Algorithmic barriers from phase transitions.
\newblock In {\em Foundations of Computer Science, 2008. FOCS'08 IEEE 49th
  Annual IEEE Symposium on}, pages 793--802. IEEE, 2008.

\bibitem{Baraniuk2008}
Richard Baraniuk, Mark Davenport, Ronald DeVore, and Michael Wakin.
\newblock A simple proof of the restricted isometry property for random
  matrices.
\newblock {\em Constructive Approximation}, 28(3):253--263, Dec 2008.

\bibitem{Barbier14}
J.~{Barbier} and F.~{Krzakala}.
\newblock Replica analysis and approximate message passing decoder for
  superposition codes.
\newblock In {\em 2014 IEEE International Symposium on Information Theory},
  pages 1494--1498, June 2014.

\bibitem{Chris17}
I.~{Ben Atitallah}, C.~{Thrampoulidis}, A.~{Kammoun}, T.~Y. {Al-Naffouri},
  M.~{Alouini}, and B.~{Hassibi}.
\newblock The box-lasso with application to gssk modulation in massive mimo
  systems.
\newblock In {\em 2017 IEEE International Symposium on Information Theory
  (ISIT)}, pages 1082--1086, June 2017.

\bibitem{BertsimasRegression}
D.~Bertsimas and Bart~Van Parys.
\newblock Sparse high dimensional regression: Exact scalable algorithms and
  phase transitions.
\newblock {\em arXiv preprint arXiv:1709.10029}, 2017.

\bibitem{BickelGenome}
Peter~J. Bickel, James~B. Brown, Haiyan Huang, and Qunhua Li.
\newblock An overview of recent developments in genomics and associated
  statistical methods.
\newblock {\em Phil. Trans. R. Soc. A}, 2009.

\bibitem{bickel2009}
Peter~J. Bickel, Ya'acov Ritov, and Alexandre~B. Tsybakov.
\newblock Simultaneous analysis of lasso and dantzig selector.
\newblock {\em Ann. Statist.}, 37(4):1705--1732, 08 2009.

\bibitem{davies}
Thomas Blumensath and Mike~E. Davies.
\newblock Iterative hard thresholding for compressed sensing.
\newblock {\em Applied and Computational Harmonic Analysis}, 27(3):265 -- 274,
  2009.

\bibitem{MRImed}
Zhao Bo, Wenmiao Lu, T.~Kevin Hitchens, Fan Lam, Chien Ho, and Zhi‐Pei Liang.
\newblock Accelerated mr parameter mapping with low‐rank and sparsity
  constraints.
\newblock {\em Magnetic Resonance in Medicine}, 2014.

\bibitem{chai}
T.~T. Cai and L.~Wang.
\newblock Orthogonal matching pursuit for sparse signal recovery with noise.
\newblock {\em IEEE Transactions on Information Theory}, 57(7):4680--4688, July
  2011.

\bibitem{Tony18}
T.~Tony Cai and Zijian Gao.
\newblock Accuracy assessment for high-dimensional linear regression1.
\newblock {\em The Annals of Statistics}, 2018.

\bibitem{candes2007}
Emmanuel Candes and Terence Tao.
\newblock The dantzig selector: Statistical estimation when p is much larger
  than n.
\newblock {\em Ann. Statist.}, 35(6):2313--2351, 12 2007.

\bibitem{candes}
Emmanuel~J. Candes, Justin~K. Romberg, and Terence Tao.
\newblock Stable signal recovery from incomplete and inaccurate measurements.
\newblock {\em Communications on Pure and Applied Mathematics},
  59(8):1207--1223, 2006.

\bibitem{candes2005decoding}
Emmanuel~J Candes and Terence Tao.
\newblock Decoding by linear programming.
\newblock {\em IEEE transactions on information theory}, 51(12):4203--4215,
  2005.

\bibitem{JASAgenomics}
Carlos~M. Carvalho, Jeffrey Chang, Joseph~E. Lucas, Joseph~R. Nevins, Quanli
  Wang, and Mike West.
\newblock High-dimensional sparse factor modeling: Applications in gene
  expression genomics.
\newblock {\em Journal of the American Statistical Association}, 2008.

\bibitem{SignDen}
Scott~Shaobing Chen, David~L. Donoho, and Michael~A. Saunders.
\newblock Atomic decomposition by basis pursuit.
\newblock {\em SIAM Rev.}, 43(1):129--159, January 2001.

\bibitem{coja2011independent}
A.~Coja-Oghlan and C.~Efthymiou.
\newblock On independent sets in random graphs.
\newblock In {\em Proceedings of the Twenty-Second Annual ACM-SIAM Symposium on
  Discrete Algorithms}, pages 136--144. SIAM, 2011.

\bibitem{DonohoAMP}
D.~L. Donoho, A.~Javanmard, and A.~Montanari.
\newblock Information-theoretically optimal compressed sensing via spatial
  coupling and approximate message passing.
\newblock {\em IEEE Transactions on Information Theory}, 59(11):7434--7464, Nov
  2013.

\bibitem{donoho2006compressed}
David~L Donoho.
\newblock Compressed sensing.
\newblock {\em IEEE Transactions on information theory}, 52(4):1289--1306,
  2006.

\bibitem{donoho2006counting}
David~L. Donoho and Jared Tanner.
\newblock Counting the faces of randomly-projected hypercubes and orthants,
  with applications.
\newblock {\em Discrete {\&} Computational Geometry}, 43(3):522--541, Apr 2010.

\bibitem{Casto11}
Emmanuel J.~Cand`es Ery Arias-Castro and Yaniv Plan.
\newblock Global testing under sparse alternatives: Anova, multiple comparisons
  and the higher criticism.
\newblock {\em The Annals of Statistics}, 2011.

\bibitem{gamarnik2016finding}
David Gamarnik and Quan Li.
\newblock Finding a large submatrix of a gaussian random matrix.
\newblock {\em arXiv preprint arXiv:1602.08529}, 2016.

\bibitem{gamarnik2014limits}
David Gamarnik and Madhu Sudan.
\newblock Limits of local algorithms over sparse random graphs.
\newblock {\em Annals of Probability. {\rm To appear}}.

\bibitem{gamarnik2014performance}
David Gamarnik and Madhu Sudan.
\newblock Performance of sequential local algorithms for the random nae-k-sat
  problem.
\newblock {\em SIAM Journal on Computing. {\rm To appear}}.

\bibitem{gamarnikzadik}
David Gamarnik and Ilias Zadik.
\newblock High dimensional linear regression with binary coefficients: Mean
  squared error and a phase transition.
\newblock {\em Conference on Learning Theory (COLT)}, 2017.

\bibitem{Ed2012}
Edward~I. George.
\newblock The variable selection problem.
\newblock {\em Journal of the American Statistical Association}, 2012.

\bibitem{hanson1971}
D.~L. Hanson and F.~T. Wright.
\newblock A bound on tail probabilities for quadratic forms in independent
  random variables.
\newblock {\em Ann. Math. Statist.}, 42(3):1079--1083, 06 1971.

\bibitem{Book15}
Trevor Hastie, Robert Tibshirani, and Martin~J. Wainwright.
\newblock {\em Statistical Learning with Sparsity: The Lasso and
  Generalizations}.
\newblock Chapman and Hall/CRC Monographs on Statistics and Applied
  Probability, 2015.

\bibitem{Lucas17}
Lucas Janson, Rina~Foygel Barber, and Emmanuel Cand{\`e}s.
\newblock Eigenprism: inference for high dimensional signal-to-noise ratios.
\newblock {\em Journal of the Royal Statistical Society. Series B}, 2017.

\bibitem{johnstone2009consistency}
Iain~M Johnstone and Arthur~Yu Lu.
\newblock On consistency and sparsity for principal components analysis in high
  dimensions.
\newblock {\em Journal of the American Statistical Association}, 104(486),
  2009.

\bibitem{Barron12}
A.~{Joseph} and A.~R. {Barron}.
\newblock Least squares superposition codes of moderate dictionary size are
  reliable at rates up to capacity.
\newblock {\em IEEE Transactions on Information Theory}, 58(5):2541--2557, May
  2012.

\bibitem{Barron14}
A.~{Joseph} and A.~R. {Barron}.
\newblock Fast sparse superposition codes have near exponential error
  probability for $r<{\cal c}$.
\newblock {\em IEEE Transactions on Information Theory}, 60(2):919--942, Feb
  2014.

\bibitem{DonohoMRI}
M.~Lustig, D.~L. Donoho, J.~M. Santos, and J.~M. Pauly.
\newblock Compressed sensing mri.
\newblock {\em IEEE Signal Processing Magazine}, 25(2):72--82, March 2008.

\bibitem{Aos06}
Nicolai Meinshausen and Peter B{\"u}hlmann.
\newblock High-dimensional graphs and variable selection with the lasso.
\newblock {\em The Annals of Statistics}, 2006.

\bibitem{montanari2011reconstruction}
Andrea Montanari, Ricardo Restrepo, and Prasad Tetali.
\newblock Reconstruction and clustering in random constraint satisfaction
  problems.
\newblock {\em SIAM Journal on Discrete Mathematics}, 25(2):771--808, 2011.

\bibitem{Van13}
Richard Nickl and Sara Van~De Geer.
\newblock Confidence sets in sparse regression.
\newblock {\em The Annals of Statistics}, 2013.

\bibitem{Chris13}
S.~{Oymak}, C.~{Thrampoulidis}, and B.~{Hassibi}.
\newblock The squared-error of generalized lasso: A precise analysis.
\newblock In {\em 2013 51st Annual Allerton Conference on Communication,
  Control, and Computing (Allerton)}, pages 1002--1009, Oct 2013.

\bibitem{CSsensor}
Bao Peng, Zhi Zhao, Guangjie Han, and Jian Shen.
\newblock Consensus-based sparse signal reconstruction algorithm for wireless
  sensor networks.
\newblock {\em International Journal of Distributed Sensor Networks}, 2016.

\bibitem{CSwireless}
Giorgio Quer, Riccardo Masiero, Gianluigi Pillonetto, Michele Rossi, and
  Michele Zorzi.
\newblock Sensing, compression, and recovery for wsns: Sparse signal modeling
  and monitoring framework.
\newblock {\em IEEE Transactions on Wireless Communications}, 2012.

\bibitem{Rad2011}
K.~Rahnama Rad.
\newblock Nearly sharp sufficient conditions on exact sparsity pattern
  recovery.
\newblock {\em IEEE Transactions on Information Theory}, 57(7):4672--4679, July
  2011.

\bibitem{rahman2014local}
Mustazee Rahman and Balint Virag.
\newblock Local algorithms for independent sets are half-optimal.
\newblock {\em arXiv preprint arXiv:1402.0485}, 2014.

\bibitem{Galen13}
Galen Reeves and Michael Gapstar.
\newblock Approximate sparsity pattern recovery: Information-theoretic lower
  bounds.
\newblock {\em IEEE Trans. Information Theory}, 2013.

\bibitem{Rush17}
C.~{Rush}, A.~{Greig}, and R.~{Venkataramanan}.
\newblock Capacity-achieving sparse superposition codes via approximate message
  passing decoding.
\newblock {\em IEEE Transactions on Information Theory}, 63(3):1476--1500,
  March 2017.

\bibitem{vandegeer2009}
Sara~A. van~de Geer and Peter B{\"u}hlmann.
\newblock On the conditions used to prove oracle results for the lasso.
\newblock {\em Electron. J. Statist.}, 3:1360--1392, 2009.

\bibitem{wainwright2009information}
Martin~J Wainwright.
\newblock Information-theoretic limits on sparsity recovery in the
  high-dimensional and noisy setting.
\newblock {\em Information Theory, IEEE Transactions on}, 55(12):5728--5741,
  2009.

\bibitem{wainwright2009sharp}
Martin~J Wainwright.
\newblock Sharp thresholds for high-dimensional and noisy sparsity recovery
  using constrained quadratic programming (lasso).
\newblock {\em IEEE transactions on information theory}, 55(5):2183--2202,
  2009.

\bibitem{wang2010information}
Wei Wang, Martin~J Wainwright, and Kannan Ramchandran.
\newblock Information-theoretic limits on sparse signal recovery: Dense versus
  sparse measurement matrices.
\newblock {\em IEEE Transactions on Information Theory}, 56(6):2967--2979,
  2010.

\bibitem{Model93}
Ping Zhang.
\newblock Model selection via multifold cross validation.
\newblock {\em The Annals of Statistics}, 1993.

\end{thebibliography}

\end{document}